\documentclass[english,11pt]{amsart}

\usepackage{latexsym}
\usepackage{multirow}
\usepackage{amsmath}
\usepackage{amssymb}
\usepackage{amsfonts}
\usepackage{mathtools} % for coloneqq
\usepackage[T1]{fontenc}
\usepackage{amsthm} \usepackage{enumerate} \usepackage[english]{babel}
\usepackage{tikz}
\usetikzlibrary{matrix,arrows}
\usepackage{tikz-cd}
\usepackage[utf8]{inputenc}
\usepackage{cite}
\usepackage{doi}

\usepackage{hyperref}

\newtheorem{theorem}{Theorem}[section]
\newtheorem{proposition}[theorem]{Proposition}
\newtheorem{corollary}[theorem]{Corollary}
\newtheorem{lemma}[theorem]{Lemma}
\theoremstyle{definition}
\newtheorem{definition}[theorem]{Definition}
\newtheorem{example}[theorem]{Example}

\theoremstyle{remark}
\newtheorem{remark}[theorem]{Remark}

\newcommand{\mb}[1]{\mathbb{#1}}

\newcommand\Hom{{\rm Hom}}

\newcommand{\im}{\operatorname{im}}
\newcommand{\rank}{\operatorname{rank}}
\newcommand{\ev}{\operatorname{ev}}
\newcommand{\xra}[1]{\xrightarrow{#1}}

\newcommand{\Spec}{\operatorname{Spec}}

\newcommand{\Pic}{\operatorname{Pic}}
\newcommand{\Sym}{\operatorname{Sym}}
\newcommand{\Cone}{\operatorname{Cone}}

\newcommand{\Cl}{\operatorname{Cl}}
\newcommand{\length}{\operatorname{length}}
\newcommand{\sat}[1]{#1^{\text{sat}}}

\newcommand{\rr}{\operatorname{r}}
\newcommand{\crr}{\operatorname{cr}}
\newcommand{\srr}{\operatorname{sr}}
\newcommand{\brr}{\operatorname{\underline{r}}}
\newcommand{\apolar}{{\: \lrcorner \:}}
\newcommand{\ccdots}{\cdot\ldots\cdot}

\begin{document}

\author{Maciej Ga\l{}\k{a}zka}
\address{Maciej Ga\l{}\k{a}zka, Faculty of Mathematics, Computer Science, and Mechanics, University of Warsaw, ul. Banacha 2, 02-097 Warszawa, Poland}
\email{mgalazka@mimuw.edu.pl}
\title{Multigraded Apolarity}
\date{\today}
% 2010 classification
%14-xx Algebraic geometry\\
%14Mxx Special varieties\\
%14M25 Toric varieties}
%14-xx Algebraic geometry\\
%14Nxx Projective and enumerative geometry\\
%14N15 Classical problems and Schubert calculus

\begin{abstract}
  We generalize methods to compute various kinds of rank to the case of a toric variety $X$ embedded into projective space using a very ample line
  bundle $\mathcal{L}$. We find an upper bound on the cactus rank. We use this to compute rank, border rank, and cactus rank of monomials in $H^0(X,
  \mathcal{L})^*$ when $X$ is $\mathbb{P}^1 \times \mathbb{P}^1$, the Hirzebruch surface $\mathbb{F}_1$, the weighted projective plane
  $\mathbb{P}(1,1,4)$, or a fake weighted projective plane.
\end{abstract}

\maketitle

\tableofcontents
\pagebreak
\section{Introduction}

\subsection{Background}

The topic of calculating ranks of polynomials goes back to works of Sylvester on apolarity in the 19th century. For introductions to this subject, see
\cite{iarrobino_kanev_book_Gorenstein_algebras} and \cite{landsberg_tensorbook}. For a concise introduction to the concept of rank for different
subvarieties $X \subseteq \mathbb{P}^N$ and numerous ways to give lower bounds for rank, see \cite{teitler_geometric_lower_bounds} (see also many
references there). For a brief review of the apolarity action in the case of the Veronese map, see \cite[Section 3]{nisiabu_jabu_cactus}.

As far we know, the notion of cactus rank was first defined in \cite[Chapter 5]{iarrobino_kanev_book_Gorenstein_algebras} (where it is called the
``scheme length''). For a motivation, basic properties and an application in the case of the Veronese embedding, see \cite{nisiabu_jabu_cactus}. We
study cactus rank, because properties of the Hilbert scheme of all zero-dimensional subschemes of a variety are better understood than properties of
the subset corresponding to smooth schemes (i.e.\ schemes of points). Another reason is that many bounds for rank work also for cactus rank, for
instance the Landsberg-Ottaviani bound for vector bundles (see \cite{mgalazka_cactus_equations}). There is also a lower bound for the cactus rank by
Ranestad and Schreyer (see \cite{ranestad_schreyer_on_the_rank_of_a_symmetric_form}).

In this paper, we see what happens when $X$ is a toric variety. For an introduction to this subject, see the newer \cite{cox_book} and the older
\cite{fulton}. For toric varieties, many invariants can be computed quite easily. This can be used to study ranks and secant varieties. In
\cite{cox_sidman}, the authors investigate the second secant variety $\sigma_2(X)$, where $X$ is a toric variety embedded into some projective space.
As they write there, ``Many classical varieties whose secant varieties have been studied are toric''. Here we take a different approach. We generalize
apolarity to toric varieties, and then, as an application, we compute rank, cactus rank and border rank of some polynomials.

\subsection{Main results}
We need to introduce some notions to state the main results. Suppose $X$ is a $\mathbb{Q}$-factorial projective toric variety. Let $S$ be the Cox ring
of $X$. By definition, it is graded by $\Cl X$. Since $X$ is a toric variety, $S$ is a polynomial ring with finitely many variables (see \cite[Section
5.2]{cox_book}), so we may write $S \cong \mathbb{C}[x_1,\dots,x_r]$. Introduce $T = \mathbb{C}[y_1,\ldots,y_r]$. We will think of $T$ as an
$S$-module, where the multiplication (denoted by $\apolar$) is induced by
\begin{equation}\label{equation:apolarity}
  x_i \apolar y_1^{b_1}\cdot \ldots \cdot y_r^{b_r} = \begin{cases}
  y_1^{b_1}\cdot\ldots\cdot y_i^{b_i-1}\cdot\ldots\cdot y_r^{b_r} & \text{if} \, b_i > 0,\\
  0 & \text{otherwise.}\end{cases}
\end{equation}
We define a grading on $T$ in $\Cl X$ in an analogous way as on $S$:
\begin{equation*}
   \deg y_1^{a_1}\ccdots y_r^{a_r} = \deg x_1^{a_1}\ccdots x_r^{a_r}\text{.}
\end{equation*}
Let $\alpha \in \Pic X$ be a very ample class. The pairing $\apolar$ gives a duality which identifies $H^0(X, \mathcal{O}(\alpha))^*$ with $T_\alpha$
(this is described in detail in Proposition \ref{proposition:duality}). Here and later, for any graded ring $R$ and any degree $\mu$, we denote by
$R_\mu$ the graded piece of $R$ of degree $\mu$. For $F \in T_\alpha$ we define $F^\perp$ as its annihilator in $S$ (with respect to the action
$\apolar$).

The first main result of this paper is:
\begin{theorem}[Multigraded Apolarity Lemma]\label{theorem:multigraded_apolarity}
  Let
  \begin{equation*}
    \varphi \colon X \hookrightarrow \mathbb{P}(H^0(X, \mathcal{O}(\alpha))^*)
  \end{equation*}
  be the morphism associated with the complete linear system $|\mathcal{O}(\alpha)|$. Fix a non-zero $F \in H^0(X, \mathcal{O}(\alpha))^*$.
  Then for any closed subscheme $R \hookrightarrow X$ we have
  \begin{equation*}
    F \in \langle R \rangle \iff I(R) \subseteq F^\perp \text{.}
  \end{equation*}
  Here $I(R)$ is the ideal of $R$ from Definition \ref{definition:ideal_of_subscheme}, and $\langle R \rangle$ is the linear span of a subscheme (see
  the beginning of Subsection \ref{subsection:cactus_rank}).
\end{theorem}
This was first proven in my master thesis (see \cite{mgalazka_master_thesis}). Then it was independently proven for smooth $X_\Sigma$ in \cite[Lemma
1.3]{toric_ranestad}.  In the paper the authors use this to determine varieties of apolar subschemes for $\mathbb{P}^1 \times \mathbb{P}^1$ embedded
into projective space by $\mathcal{O}(2,2)$ and $\mathcal{O}(3,3)$, and also for the Hirzebruch surface $\mathbb{F}_1$ embedded by the bundle
$\mathcal{O}(2,1)$ (in notation from Subsection \ref{subsection:hirzebruch_surface}).

We prove Theorem \ref{theorem:multigraded_apolarity} in Section \ref{section:apolarity_lemma}. 

Suppose $\beta \in \Cl X$. Consider the restriction of the action $\apolar$ to
\begin{equation*}
  S_{\beta} \times T_{\alpha} \xra{\apolar} T_{\alpha - \beta} \text{.}
\end{equation*}
For any $F \in T_\alpha$ we consider the linear map $C_F^\beta : S_\beta \to T_{\alpha - \beta}$ given by $h \mapsto h\apolar F$.
\begin{theorem}\label{theorem:catalecticant}
  Fix $F \in H^0(X, \mathcal{O}(\alpha))^*$. We have the following:
  \begin{enumerate}[(1)]
    \item if $\beta \in \Cl X$, then
    \begin{equation*}
      \brr(F) \geq \rank(C_F^\beta)\text{,}
    \end{equation*}
    \item if $\beta \in \Pic X$, then
    \begin{equation*}
      \crr(F) \geq \rank(C_F^\beta)\text{.}
    \end{equation*}
  \end{enumerate}
  Here $\brr(F)$ and $\crr(F)$ denote the border rank and the cactus rank of $F$, respectively, see Definitions
  \ref{definition:sigma_rank_border_rank} and \ref{definition:cactus_rank}.
\end{theorem}
We also provide an example such that the bound in point (1) does not hold for the cactus rank, see Remark \ref{remark:catalecticant_counterexample}.

Theorem \ref{theorem:catalecticant} is proven in Corollary \ref{corollary:cactus_catalecticant_bound} and Corollary
\ref{corollary:border_catalecticant_bound}. The bound in point (2) was given in \cite[Theorem 5.3.D]{iarrobino_kanev_book_Gorenstein_algebras} for the
Veronese embedding. Also see \cite{mgalazka_cactus_equations} for a version of the bound in point (2) for vector bundles of higher rank. 

As an application of Theorem \ref{theorem:multigraded_apolarity}, in Section \ref{section:upper_bound_on_cactus} we provide an upper bound for the 
cactus rank of a polynomial. 
\begin{theorem} \label{theorem:upper_bound_on_cactus}
  Suppose $X$ is a smooth projective toric variety, and $\alpha \in \Pic X$ is a very ample class. Let $0 \neq F \in H^0(X, \mathcal{O}(\alpha))^*$.
  Let $\sigma$ be any cone of the fan of $X$ of maximal dimension. Let $f$ be the dehomogenization of $F$ defined by setting all the variables
  corresponding to rays not in $\sigma$ to $1$. Then
  \begin{equation*}
    \crr(F) \leq \dim S/f^\perp\text{.}
  \end{equation*}
\end{theorem}
This theorem is a generalization of \cite[Theorem 3]{bernardi_ranestad_cactus_rank_of_cubics} to the multigraded setting.
From it we derive a corollary.
\begin{corollary}\label{corollary:upper_bound_segre_veronese}
  Let 
  \begin{equation*}
    \mathbb{P}^{n_1}\times \dots \times \mathbb{P}^{n_k} \xra{v_{d_1,\dots,d_k}} \mathbb{P}(\Sym^{d_1}\mathbb{C}^{n_1 + 1}\otimes \dots \otimes
    \Sym^{d_k}\mathbb{C}^{n_k + 1})
  \end{equation*}
  be a Segre-Veronese embedding. Here $\Sym^i$ denotes the $i$-th symmetric tensors. Let $F \in \Sym^{d_1}\mathbb{C}^{n_1 + 1}\otimes \dots \otimes
  \Sym^{d_k}\mathbb{C}^{n_k + 1}$ be a non-zero form. Then
  \begin{align*}
    \crr(F) &\leq \sum_{\substack{(e_1,\dots,e_k) |\\ e_1 + \dots + e_k \leq d/2 }}\binom{n_1 -1+ e_1}{e_1}\ccdots \binom{n_k -1+ e_k}{e_k} \\
    &+\sum_{\substack{(e_1,\dots,e_k) |\\ e_1 + \dots + e_k > d/2 }}\binom{n_1 -1+ d_1 - e_1}{d_1 - e_1}\ccdots \binom{n_k -1 + d_k - e_k}{d_k - e_k}\text{.}
  \end{align*}
\end{corollary}
In \cite{ballico_bernardi_gesmundo_cactus_rank_segre_veronese} the authors prove a weaker version of the bound in Corollary
\ref{corollary:upper_bound_segre_veronese}.

See Section \ref{section:upper_bound_on_cactus} for the proofs of Theorem \ref{theorem:upper_bound_on_cactus} and Corollary
\ref{corollary:upper_bound_segre_veronese}.

Finally, we use this to compute ranks of monomials when $X$ is a projective toric surface. The first example is $\mathbb{P}^1 \times \mathbb{P}^1$,
see Subsection \ref{subsection:p1timesp1}. We consider the problem of determining cactus ranks and ranks of monomials $F =
x_0^{k_0}x_1^{k_1}y_0^{l_0}y_1^{l_1}$, where $k_0 \geq k_1 \geq 1, l_0 \geq l_1 \geq 1$. We have
\begin{equation}\label{equation:obvious_rank_inequality}
   \rr(F) \leq  (k_0 + 1)(l_0 + 1)\text{.}
\end{equation}
But the equality in the equation above does not always hold. For example, rank of $x_0^2 x_1 y_0^2 y_1$ is $8$, not $9$ (see \cite[Remark
16]{christandl_kjaerulff_zuiddam_tensor_rank_not_multiplicative}, \cite{chen_friedland_rank_of_tensor_product_is_eight}).

Our result is 
\begin{theorem}\label{theorem:inequalities}
  The following inequalities hold:
  \begin{enumerate}[(i)]
    \item $\rr(F) \leq (k_0 + 1)(l_1 + 1) + (k_1 + 1)(l_0 + 1) - (k_1 + 1)(l_1 + 1)$,\label{item:first_inequality}
    \item $\rr(F) \geq (k_0 + 1)(l_1 + 1)$ for $k_0 > k_1$, $\rr(F) \geq (k_1 + 1)(l_0 + 1)$ for $l_0 > l_1$,\label{item:second_inequality}
    \item $\rr(F) \geq (k_1 + 2)(l_1 + 2) - 1$ for $k_0 > k_1$ and $l_0 > l_1$.\label{item:third_inequality}
  \end{enumerate}
\end{theorem}

Item \eqref{item:first_inequality} is stronger than the recent result \cite[Proposition
3.9]{ballico_bernardi_christandl_gesmundo_partially_symmetric}. Item \eqref{item:second_inequality} is proven independently in \cite[Proposition 4.3]{
ballico_bernardi_gesmundo_oneto_ventura_geometric_conditions}.

Let us look at the cases where rank is determined by these inequalities. When we set $l_1 = l_0$ in the first inequality of Item
\eqref{item:second_inequality}, from Equation \eqref{equation:obvious_rank_inequality} we get $\rr(F) = (k_0 + 1)(l_0 + 1)$. Also when we set $k_0 =
k_1 + 1$ and $l_0 = l_1 + 1$, we get (by Items \eqref{item:first_inequality} and \eqref{item:third_inequality}) that $\rr(F) = (k_1 + 2)(l_1 + 2) -
1$. 

The next example is the Hirzebruch surface $\mathbb{F}_1$ (which can be defined as $\mathbb{P}^2$ blown up in one point), see Subsection
\ref{subsection:hirzebruch_surface}. Here we find monomials whose border rank is less than their cactus rank (and also their smoothable rank), see
Remark \ref{remark:wild_case}. Another one is the weighted projective plane $\mathbb{P}(1,1,4)$, see Subsection
\ref{subsection:weighted_projective_plane}. Here we give an example of a monomial whose cactus rank is less than its border rank. The last one is a
fake weighted projective plane (see Subsection \ref{subsection:fake_weighted_projective_plane}) --- the quotient of $\mathbb{P}^2$ by the action of
$\mathbb{Z}/3 = \{1, \varepsilon, \varepsilon^2 \}$ (where $\varepsilon^3 = 1$) given by $\varepsilon \cdot [\lambda_0,\lambda_1,\lambda_2] =
[\lambda_0, \varepsilon\lambda_1,\varepsilon^2 \lambda_2]$. 

\subsection{Acknowledgments}
This article a severely expanded version of my master thesis, \cite{mgalazka_master_thesis}.

I thank my advisor, Jaros\l{}aw Buczy\'nski, for introducing me to this subject, his insight, many suggestions of examples, suggestions on how to
improve the presentation, many discussions, and constant support. I also thank Piotr Achinger and Joachim Jelisiejew for suggestions on how to improve
the presentation. I am also grateful to Joachim Jelisiejew and Mateusz Micha{\l}ek for helpful discussions.

I was supported by the project ``Secant varieties, computational complexity, and toric degenerations'' realized withing the Homing Plus programme of
Foundation for Polish Science, co-financed from European Union, Regional Development Fund, and by Warsaw Center of Mathematics and Computer Science
financed by Polish program KNOW. During the process of expanding the article (adding Section \ref{section:upper_bound_on_cactus} and Subsection
\ref{subsection:p1timesp1}) I was supported by the NCN project ``Algebraic Geometry: Varieties and Structures'' no. 2013/08/A/ST1/00804.

\section{Ranks and secant varieties}
In this section we review the definitions of various kinds of ranks and secant varieties.
\begin{definition}\label{definition:sigma_rank_border_rank}
  Let $W$ be a finite-dimensional complex vector space, and $X$ a subvariety of $\mathbb{P}W$. Let
  \begin{equation*}
    \sigma_r^0(X) = \{ [F] \in \mathbb{P}W | [F] \in \langle p_1,\dots, p_r \rangle \text{ where } p_1,\dots, p_r \in X \}\text{,}
  \end{equation*}
  where $\langle \rangle$ denotes the (projective) linear span. Define the $r$-th \emph{secant variety} of $X \subseteq \mathbb{P}W$ by $\sigma_r(X) =
  \overline{\sigma_r^0(X)}$. The overline denotes the Zariski closure. For any non-zero $F \in W$ define the
  \emph{$X$-rank} of $F$:
  \begin{align*}
    \rr_X(F) &= \min \{ r \in \mathbb{Z}_{\geq 1} | [F] \in \sigma_r^0(X) \} \\
             &= \min \{ r \in \mathbb{Z}_{\geq 1} | [F] \in \langle p_1,\dots,p_r \rangle \text{ for some } p_1,\dots,p_r \in X\}
  \end{align*}
  and the \emph{$X$-border rank} of $F$:
  \begin{align*}
    \brr_X(F) &= \min \{ r \in \mathbb{Z}_{\geq 1} | [F] \in \sigma_r(X) \}\\
    &= \min \{ r \in \mathbb{Z}_{\geq 1} | F \text{ is a limit of points of }X\text{-rank } \leq r \}\text{.} 
  \end{align*}
  Usually, if $X$ is fixed, we omit the prefix and call them rank and border rank, respectively. 
\end{definition}

The problem of calculating border rank of points is related to finding equations of secant varieties. Namely, if we know set-theoretic equations of
$\sigma_r(X) \subseteq \mathbb{P}W$ for $r = 1,2,3,\dots$, then we can calculate the border rank of any point (by checking if it satisfies the
equations).

\begin{example}
  Let $X$ be the $d$-th Veronese variety $\mathbb{P}V \subseteq \mathbb{P}\Sym^d V$. Then $X$-rank of $[F] \in \mathbb{P} \Sym^d V$ is the least $r$
  such that $F$ can be written as $v_1^d + \dots + v_r^d$ for some $v_i \in V$. The $X$-rank is called the symmetric rank, or the Waring rank in this
  case.
\end{example}

Let us go back to the setting of a projective variety $X \subseteq \mathbb{P}W$. Here are a few results which we are going to need later.

Let $\mathbb{P}T_{q}X$ denote the projective tangent space of $X$ embedded in $\mathbb{P}W$ at point $q$, i.e.\ the projectivization of the affine
tangent space to the affine cone of $X$.
\begin{proposition}[Terracini's Lemma]\label{proposition:terracini}
  Let $r$ be a positive integer. Then for $r$ general points $p_1,\dots, p_r \in X$ and a general point $q \in \langle p_1,\dots,p_r \rangle$ we have
  \begin{equation*}
    \mathbb{P}T_{q}\sigma_r(X) = \langle \mathbb{P}T_{p_1} X,\dots,\mathbb{P}T_{p_r} X \rangle\text{.}
  \end{equation*}
\end{proposition}
For a proof, see \cite[Section 5.3]{landsberg_tensorbook} or \cite[Chapter V, Proposition 1.4]{zak_tangents}.
\begin{corollary}[of Proposition \ref{proposition:terracini}]
  The dimension of $\sigma_r(X)$ is not greater than $r(\dim X + 1) - 1$.
\end{corollary}
\begin{proposition}
  If $X$ is irreducible, then $\sigma_r(X)$ is irreducible for any $r \geq 1$.
\end{proposition}
\begin{definition}\label{definition:expected_dimension}
  When $\dim \sigma_r(X) = \min(\dim \mathbb{P}W, r(\dim X + 1) -1)$, we say that $\sigma_r(X)$ is of \emph{expected dimension}.
\end{definition}

\subsection{Cactus rank}\label{subsection:cactus_rank}
For a zero-dimensional scheme $R$ (of finite type over $\mathbb{C}$), let $\length{R}$ denote its length, i.e.\ $\dim_{\mathbb{C}}H^0(R,
\mathcal{O}_R)$.  This is equal to the degree of $R$ in any embedding into projective space. Also for any subscheme $R \hookrightarrow \mathbb{P}W$
define $\langle R \rangle$ to be the linear span of $R$, i.e.\ the smallest projective linear space, through which the inclusion of the scheme factors.

\begin{definition}\label{definition:cactus_rank}
  Define the \emph{$X$-cactus rank} of $F \in W$: 
  \begin{equation*}
    \crr_X(F) = \min \{\length R | R \hookrightarrow X, \dim R = 0, F \in \langle R \rangle \}\text{.} \\
  \end{equation*}
\end{definition}
We have the following inequalities:
\begin{align*}
  \crr(F) &\leq \rr(F)\text{,} \\
  \brr(F) &\leq \rr(F)\text{.} 
\end{align*}

\section{Toric varieties}\label{section:toric_varieties}
\subsection{Quotient construction and the Cox ring}
Let $M$ and $N$ be dual lattices (abelian groups isomorphic to $\mathbb{Z}^k$ for some $k \geq 1$) and $\langle \cdot,\cdot\rangle :M\times N\to
\mathbb{Z}$ be the duality between them. Let $X_\Sigma$ be the toric variety of a fan $\Sigma \subseteq N_\mathbb{R} \coloneqq N\otimes \mathbb{R}$
with no torus factors. The term ``with no torus factors'' means that the linear span of $\Sigma$ in $N_\mathbb{R}$ is the whole space. Let $\Sigma(1)$
denote the set of rays of the fan $\Sigma$. Similarly, $\sigma(1)$ denotes the set of rays in a cone $\sigma$. Then $X_\Sigma$ can be obtained as an
almost geometric quotient of an action of $G \coloneqq \Hom(\Cl X_\Sigma, \mathbb{C}^*)$ on $\mathbb{C}^{\Sigma(1)}\setminus Z$, where $Z$ is a
subvariety of $\mathbb{C}^{\Sigma(1)}$. Let us go briefly through the construction of this quotient. We follow \cite[Section 5.1]{cox_book}. 

Since $X_\Sigma$ has no torus factors, we have an exact sequence
\begin{equation*}
  0 \to M \to \mathbb{Z}^{\Sigma(1)} \to \Cl X_\Sigma \to 0\text{.}
\end{equation*}
After applying $\Hom(-, \mathbb{C}^*)$, this gives
\begin{equation*}
  1 \to \Hom(\Cl X_\Sigma, \mathbb{C}^*) \to (\mathbb{C}^*)^{\Sigma(1)} \to \mathbb{C}^*\otimes N \to 1\text{.}
\end{equation*}
So $G = \Hom(\Cl X_\Sigma, \mathbb{C}^*)$ is a subset of $\mathbb{C}^{\Sigma(1)}$, and the action is given by multiplication on coordinates. Let $S =
\Spec \mathbb{C}[x_\rho | \rho \in \Sigma(1)]$. In other words, $S$ is the polynomial ring with variables indexed by the rays of the fan $\Sigma$. The
ring $S$ is the coordinate ring of the affine space $\mathbb{C}^{\Sigma(1)}$. For a cone $\sigma \in \Sigma$, define
\begin{equation*}
  x^{\widehat{\sigma}} = \prod_{\rho \notin \sigma(1)}{x_\rho}\text{.}
\end{equation*}
Then define a homogeneous ideal in $S$:
\begin{equation}\label{equation:irrelevant_ideal}
  B = B(\Sigma) = (x^{\widehat{\sigma}} | \sigma \in \Sigma) \subseteq S\text{,}
\end{equation}
which is called the irrelevant ideal, and let $Z = Z(\Sigma) \subseteq \mathbb{C}^{\Sigma(1)}$ be the vanishing set of $B$. For a precise construction
of the quotient map $[\cdot]$
\[\begin{tikzcd}
    \mathbb{C}^{\Sigma(1)}\setminus Z \arrow{r}{[\cdot]} & (\mathbb{C}^{\Sigma(1)}\setminus Z)// G = X_\Sigma
\end{tikzcd}\]
see \cite[Proposition 5.1.9]{cox_book}, where it is denoted by $\pi$.

Fix an ordering of all the rays of the fan, let $\Sigma(1) = \{\rho_1,\ldots,\rho_r\}$. Then $S$ becomes $\mathbb{C}[x_{\rho_1},\dots,x_{\rho_r}] =:
\mathbb{C}[x_1,\dots,x_r]$. The ring $S$ is the Cox ring of $X_\Sigma$. For more details, see \cite[Section 5.2]{cox_book}, where $S$ is called the total
coordinate ring. This ring is graded by the class group $\Cl{X_\Sigma}$, where
\begin{equation*}
  \deg x_i = [D_{\rho_i}]\text{,}
\end{equation*}
and $D_{\rho_i}$ is the torus-invariant divisor corresponding to $\rho_i$, see \cite[Chapter 4]{cox_book}. 

\subsection{Saturated ideals}
Take any ideals $I, J \subseteq S$. Let $(I :_S J)$ be the set of all $x \in S$ such that $x \cdot J \subseteq I$; it is an ideal of $S$. It is
sometimes called the quotient ideal, or the colon ideal. For any ideals $I, J, K \subseteq S$ we have:
\begin{itemize}
  \item $I \subseteq (I :_S J)$,
  \item if $J \subseteq K$, then $(I:_S J) \supseteq (I:_S K)$,
  \item $(I :_S J\cdot K) = ((I:_S J):_S K)$.
\end{itemize}
Recall the irrelevant ideal $B \subseteq S$ defined in Equation \eqref{equation:irrelevant_ideal}. Take any ideal $I \subset S$. We define the
$B$-saturation of $I$ as
\begin{equation*}
  \sat{I} = \bigcup_{i\geq 1}{(I:_S B^i)}\text{.}
\end{equation*}
Note that this is an increasing union because $B^i \supseteq B^j$ for $i < j$, so $\sat{I}$ is an ideal. Since $S$ is
Noetherian, the union stabilizes in a finite number of steps. We always have $I \subseteq \sat{I}$. If this is an equality, we say that $I$ is
$B$-saturated. In order to show that $I$ is $B$-saturated, it suffices to find any $i \geq 1$ such that $I = (I :_S B^i)$.

Moreover, if $I$ and $J$ are homogeneous, then so is $(I:_S J)$. It follows that for $I$ homogeneous the ideal $\sat{I}$ is homogeneous.

\begin{example}
  Let us look at the projective space $\mathbb{P}_{\mathbb{C}}^k$. See \cite[Example 5.1.7]{cox_book}. Here $S =
  \mathbb{C}[x_0,\dots,x_k]$, $B = (x_0,\dots,x_k) = \bigoplus_{i \geq 1}S_i$ and $Z = \{0\}$. In this case
  \begin{equation*}
    \sat{I} = \{f \in S | \text{ for all } i = 0,1,\dots,k \text{ there is } n \text{ such that } x_i^n\cdot f \in I \}\text{.}
  \end{equation*}
  Recall that in this case there is a 1-1 correspondence between closed subschemes of $\mathbb{P}_{\mathbb{C}}^k$ and homogeneous $B$-saturated ideals
  of $S$. Moreover, the ideal given by $\bigoplus_{i\geq 0}H^0(X, \mathcal{I}_R\otimes \mathcal{O}(i))$, where $\mathcal{I}_R$ is the ideal sheaf of
  $R$ in $\mathbb{P}_{\mathbb{C}}^k$, is $B$-saturated. For more on this, see \cite[II, Corollary 5.16 and Exercise 5.10]{hartshorne}.
\end{example}

For a toric variety the situation is more complicated. We will assume that the fan $\Sigma$ is simplicial for technical reasons. There can be many
$B$-saturated ideals defining a subscheme $R$. But they have to agree in the $\Pic$ part. See \cite[Theorem 3.7 and the following
discussion]{cox_homogeneous} for more details. Consider the map 
\begin{equation}
  \bigoplus_{\alpha \in \Cl X_\Sigma}H^0(X, \mathcal{I}_R\otimes \mathcal{O}(\alpha)) \to \bigoplus_{\alpha \in \Cl X_\Sigma}H^0(X, \mathcal{O}(\alpha))
  \label{map:saturation}
\end{equation}
induced by $\mathcal{I}_R \hookrightarrow \mathcal{O}_{X_\Sigma}$. We may take $I(R)$ to be the image of this map. This is done in the proof of
\cite[Proposition 6.A.6]{cox_book}. Note that in this case for any $\alpha \in \Pic X_\Sigma$ the vector space $H^0(X_\Sigma,
\mathcal{I}_R\otimes\mathcal{O}(\alpha))$ can be identified with those global sections of $\mathcal{O}(\alpha)$ which vanish on $R$. So let us make
the following
\begin{definition}\label{definition:ideal_of_subscheme}
  Let $X_\Sigma$ be a simplicial toric variety. Let $R \hookrightarrow X_\Sigma$ be a closed subscheme. We define $I(R) \subseteq S$, the ideal of
  $R$, to be the image of homomorphism \eqref{map:saturation}.
\end{definition}
\begin{proposition}
  \label{proposition:ideal}
  Suppose the fan $\Sigma$ is simplicial. Let $\alpha \in \Pic X_\Sigma$ be the class of a Cartier divisor. Let $R \hookrightarrow X_\Sigma$ be any
  closed subscheme. Then $(I(R))_\alpha = (I(R):_S B^i)_\alpha$ for any $i \geq 1$ (hence $I(R)$ agrees with $\sat{I(R)}$ in degree $\alpha$).
\end{proposition}
\begin{proof}
  Take $x \in S_\alpha$ such that $x \cdot B^i \subseteq I(R)$. It is enough to show that $x$ is zero on $R$. Take any point $p \in R$. We will show
  that $x$ is zero on $R$ around that point. Since the vanishing set of $B$ is empty, we know that some homogeneous element $b \in B$ is non-zero at
  $p$. By taking a large enough power, we may assume $b \in (B^i)_\beta$ for some $\beta \in \Pic X_\Sigma$ (here we use that $\Sigma$ is
  simplicial!). Because $b$ is non-zero at $p$, there is an open neighbourhood $p \in U \subseteq X_\Sigma$ such that $\mathcal{O}_{X_\Sigma}(\beta)$
  is trivialized on $U$ by $b$. But then $x$ is zero when pulled back to $R$ on $U$ if and only if $x \cdot b$ is zero when pulled back to $R$ on $U$.
  But the latter thing is true as $x \cdot b \in I(R)$.
\end{proof}

\subsection{Isomorphism between sections and polynomials}
Let $\alpha \in \Cl X_\Sigma$. Recall the isomorphism of $H^0(X_\Sigma,\mathcal{O}(\alpha))$ and
$\mathbb{C}[x_1,\ldots,x_r]_\alpha$ given in \cite[Proposition 5.3.7]{cox_book}. 

\begin{proposition}\label{proposition:isomorphism}
  Suppose $\alpha \in \Pic X_\Sigma$. Take any section $s \in H^0(X_\Sigma, \mathcal{O}(\alpha))$ and the corresponding polynomial $f \in S_\alpha$.
  Also let $p$ be a point in $X_\Sigma$ and take any $(\lambda_1,\ldots,\lambda_r) \in \mathbb{C}^r$ such that $[\lambda_1,\dots,\lambda_r] = p$.
  Then
  \begin{equation*}
    s(p) = 0 \iff f(\lambda_1,\dots,\lambda_r) = 0\text{.}
  \end{equation*}
\end{proposition}
\begin{proof}
  Take any $\sigma$ such that $p \in U_\sigma$. We will trivilize the line bundle $\mathcal{O}(\alpha)$ on $U_\sigma$ in order to move the situation
  to regular functions on $U_\sigma$. We will do it by finding a section that is nowhere zero both as a polynomial and as a section.

  We know that $U_\sigma = \Spec (S_{x^{\widehat{\sigma}}})_0$, where $x^{\widehat{\sigma}}=\prod_{\rho \notin \sigma}{x_\rho}$, the inner subscript
  refers to localization, and the outer one is taking degree $0$. From the definition of $\mathcal{O}(\alpha)$ we have
  $H^0(U_\sigma,\mathcal{O}(\alpha)) = (S_{x^{\widehat{\sigma}}})_\alpha$. Our goal is to find a monomial in $(S_{x^{\widehat{\sigma}}})_\alpha$ which
  is nowhere zero as a section.  Take any torus-invariant representative $\sum_{\rho}{a_\rho D_{\rho}}$ of class $\alpha$ (here $a_\rho \in
  \mathbb{Z}$). From \cite[Theorem 4.2.8]{cox_book} there exists an $m_\sigma \in M$ such that $\langle m_{\sigma},u_\rho\rangle = -a_\rho$ for $\rho
  \in \sigma(1)$ (here $M$ is the lattice of characters as in the beginning of Section \ref{section:toric_varieties}, $\sigma(1)$ is the set of rays of
  the cone $\sigma$, and $u_\rho \in N$ is the generator of ray $\rho$). Then
  \begin{equation*}
    \sum_{\rho}{\langle m_\sigma,u_\rho \rangle D_\rho} +
    \sum_{\rho}a_\rho D_\rho = \sum_{\rho \notin \sigma(1)}(\langle m_\sigma,u_\rho \rangle + a_\rho) D_\rho
  \end{equation*}
  belongs to the class $\alpha$ as well. This is a direct consequence of the exact sequence \cite[Theorem 4.2.1]{cox_book}. The outcome is that the
  monomial
  \begin{equation}\label{equation:monomial}
    g \coloneqq \prod_{\rho \notin \sigma(1)}x_\rho^{\langle m_\sigma,u_\rho \rangle + a_\rho}
  \end{equation}
  has degree $\alpha$.  Notice that it belongs
  to $(S_{x^{\widehat{\sigma}}})_\alpha$. 

  We want to show that $g$ is nowhere zero as a section of $\mathcal{O}(\alpha)$. The polynomial $g \in S_{x^{\widehat{\sigma}}}$ is invertible, with
  inverse $g^{-1} \in (S_{x^{\widehat{\sigma}}})_{-\alpha}$. But then $g^{-1}\cdot g = 1 \in (S_{x^{\widehat{\sigma}}})_0$. If $g$ were zero at some
  point $p \in X_\Sigma$, then we would have $0 = g^{-1}(p)\cdot g(p) = 1$, a contradiction. An analogous proof shows that $g$ is nowhere zero as a
  polynomial.

  Now we can set $\bar{f} = g^{-1}f$ and then $\bar{f}$ is a regular function on $\Spec(S_{x^{\widehat{\sigma}}})_0$.  We need to see that $\bar{f}(p)
  = 0$ is equivalent to $\bar{f}(\lambda_1,\dots,\lambda_r) = 0$. In fact, even more is true: $\bar{f}(p) = \bar{f}(\lambda_1,\dots,\lambda_r)$. To
  see this, consider the projection $\mathbb{C}^r \setminus Z \xra{[\cdot]} X_\Sigma$ restricted to the inverse image of $U_\sigma$. This corresponds
  to the homomorphism of algebras $[\cdot]^*_\sigma : \mathbb{C}[\sigma^\vee \cap M] \to S_{x^{\widehat{\sigma}}}$ given by
  \begin{equation*}
    \chi^m \mapsto \prod_{\rho \in \Sigma(1)} x_\rho^{\langle m, u_\rho\rangle}\text{,}
  \end{equation*}
  see \cite[Proof of Theorem 5.1.11]{cox_book}. Here $\sigma^\vee$ is the dual cone, $\chi^m$ is the character corresponding to $m$, and $\langle
  \cdot,\cdot \rangle$ is the standard pairing between $M$ and $N$. Let us look at the following diagram
  \[\begin{tikzcd}
      \mathbb{C}[\sigma^\vee \cap M] \ar{r}{[\cdot]_\sigma^*}\ar[rd, "\ev_p",swap]
      & S_{x^{\widehat{\sigma}}} \ar[d, "\ev_\lambda"] \\
      & \mathbb{C}\text{,}
  \end{tikzcd}\]
  where $\ev$ denotes the evaluation. When we apply the functor $\Spec$ to the diagram, it becomes commutative (since $[\lambda] = p$). As $\Spec$ is
  an equivalence of categories, the original diagram is commutative. But this means that $\bar{f}(p) = \bar{f}(\lambda_1,\dots,\lambda_r)$, as
  desired.
\end{proof}

\begin{corollary}\label{corollary:isomorphism}
  Suppose $\alpha \in \Pic X_\Sigma$. Suppose $f_1, f_2 \in S_\alpha$ are polynomials and $s_1, s_2$ are the corresponding sections of
  $\mathcal{O}(\alpha)$.  Also fix, as above, $p \in X_\Sigma$ and $(\lambda_1,\dots,\lambda_r) \in \mathbb{C}^r$ such that
  $[\lambda_1,\dots,\lambda_r] = p$. Then if $f_2(\lambda_1,\dots,\lambda_r)$ and $s_2(p)$ are non-zero, we get
  \begin{equation*}
    \frac{f_1(\lambda_1,\dots,\lambda_r)}{f_2(\lambda_1,\dots,\lambda_r)} = \frac{s_1(p)}{s_2(p)}\text{.}
  \end{equation*}
\end{corollary}
\begin{proof}
  Take $\mu \in \mathbb{C}$ such that $f_1(\lambda_1,\dots,\lambda_r) = \mu f_2(\lambda_1,\dots,\lambda_r)$. Then use the previous fact for $f_1 - \mu f_2$
  and the corresponding section $s_1 - \mu s_2$.
\end{proof}

\subsection{Generators of the class group}\label{subsection:freeness_of_the_class_group}
\begin{proposition}\label{proposition:basis_of_class_group}
  Let $X_\Sigma$ be a smooth complete variety. Pick any $\sigma \in \Sigma$ of full dimension. Let $\rho_1,\dots,\rho_d$ be the rays that are not in
  $\sigma$. Then the classes $[D_{\rho_1}],\dots,[D_{\rho_d}]$ are a basis of the class group.
\end{proposition}
\begin{proof}
  In the proof of Proposition \ref{proposition:isomorphism}, given a maximal cone $\sigma \in \Sigma$, and a class $\alpha \in \Pic X_\Sigma$, we
  constructed a monomial $g$ of degree $\alpha$, such that none of its rays belonged to $\sigma$ (see Equation \eqref{equation:monomial}). This means
  that the classes generate the class group (we use here that for smooth varieties $\Pic X_\Sigma = \Cl X_\Sigma$).

  Now consider the exact sequence
  \begin{equation*}
    0 \to M \to \mathbb{Z}^{\Sigma(1)} \to \Cl X_\Sigma \to 0\text{.}
  \end{equation*}
  From this we get that $\rank \Cl X_\Sigma = \#\Sigma(1) - \dim M_\mathbb{C}$, which is equal to the number of rays not in $\sigma$, since the cone
  $\sigma$ is smooth. But now from the exact sequence
  \begin{equation*}
    0 \to \mathbb{Z}^{l} \to \bigoplus_{\rho \notin \sigma(1)}\mathbb{Z}[D_{\rho}] \to \Cl X_\Sigma \to 0
  \end{equation*}
  we get that $l$ = 0, so
  \begin{equation*}
    \Cl X_\Sigma \cong \bigoplus_{\rho \notin \sigma(1)}\mathbb{Z}[D_{\rho}]\text{,}
  \end{equation*}
  as desired.

\end{proof}

\subsection{Dehomogenization and homogenization}\label{subsection:dehomo_and_homo}

Let $X_\Sigma$ be a smooth projective toric variety. Denote the rays of the fan by $\rho_1,\dots,\rho_r$.  Fix $\sigma \in \Sigma$. We want to
restrict $X_\Sigma$ to the affine patch $U_\sigma$. Suppose the rays that are not in $\sigma$ are $\rho_1,\dots, \rho_k$. Then the restriction
corresponds to setting $x_1,\dots, x_k$ to $1$. Denote by $\pi$ the dehomogenization on $T$ (i.e.\ setting every power $y_i^d$ to $1$) and by $\pi^*$
the dehomogenization on $S$ (i.e.\ setting $x_1,\dots,x_k$ to $1$).

\begin{proposition}\label{proposition:injectivity}
  For any $\alpha \in \Cl X_\Sigma = \Pic X_\Sigma$ the map
  \begin{equation*}
    \pi : T_\alpha \to T
  \end{equation*}
  is injective. So is the map
  \begin{equation*}
    \pi^* : S_\alpha \to S \text{.}
  \end{equation*}
\end{proposition}
\begin{proof}
  Suppose $\pi(F) = 0$ for some $0 \neq F \in T_\alpha$. Then there exist two different monomials $y_1^{c_1} \ccdots y_r^{c_r}$ and $y_1^{d_1}\ccdots
  y_r^{d_r}$ of degree $\alpha$ such that after applying $\pi$ they are the same. This means that $c_{k+1} = d_{k+1},\dots,c_r = d_r$, so $\deg
  y_1^{c_1}\ccdots y_k^{c_k} = \deg y_1^{d_1} \ccdots y_k^{d_k}$. The tuples $(c_1,\dots,c_k)$ and $(d_1,\dots,d_k)$ are different, so this gives a
  non-trivial relation between the classes corresponding to $y_1,\dots,y_k$, contradicting Proposition \ref{proposition:basis_of_class_group}.

  A similar proof works for $\pi^*$.
\end{proof}

Now let us define the homogenization $f^\text{h}$ of a non-zero polynomial $f \in \mathbb{C}[x_{k + 1},\dots, x_r]$. Suppose
\begin{equation*}
  f = \sum_{\alpha \in \Cl X_\Sigma}{f_\alpha}\text{,}
\end{equation*}
where each $f_\alpha$ is homogeneous of degree $\alpha$. Let $D_i$ be the divisor corresponding to $\rho_i$. From Proposition
\ref{proposition:basis_of_class_group} we know that the classes $[D_i]$, where $i = 1,\dots,k$, form a basis of the class group. Hence, for each
$\alpha$ such that $f_\alpha \neq 0$ we have
\begin{equation*}
  \alpha = a_{\alpha,1} [D_1] + \ldots + a_{\alpha,k} [D_k] \text{,}
\end{equation*}
where $a_{\alpha,i} \in \mathbb{Z}$. Let $b_i = \max\{a_{\alpha,i} | \alpha \in \Cl X_\Sigma, f_\alpha \neq 0\}$. Then we set
\begin{equation*}
  f^\text{h} = \sum_{\alpha \in \Cl X_\Sigma } {x_1^{b_1 - a_{\alpha,1}}\ccdots x_k^{b_k - a_{\alpha,k}} f_\alpha}\text{.}
\end{equation*}
This is homogeneous of degree $b_1 [D_1] +\ldots + b_k [D_k]$.
\begin{proposition}
  \label{proposition:hom_dehom}
  Suppose $f \in S$ is homogeneous and non-zero and let $f = x_1^{e_1}\ccdots x_k^{e_k} \hat{f}$, where the $e_i$ are natural, and $\hat{f}$ is not
  divisible by any of the $x_i$. Then $(\pi^*(f))^\text{h} = \hat{f}$.
\end{proposition}
\begin{proof}
  We know that $\pi^*(f)$ is non-zero from Proposition \ref{proposition:injectivity}. Let $\pi^*(f) = \sum_{\alpha \in \Cl X_\Sigma} g_\alpha$, where
  each $g_\alpha$ has degree $\alpha$. Then
  \begin{equation*}
    f = x_1^{e_1}\ccdots x_k^{e_k} \cdot \left( \sum_{\alpha \in \Cl X_\Sigma}{x_1^{d_{\alpha,1}}\ccdots x_r^{d_{\alpha,r}} g_\alpha}\right)
  \end{equation*}
  for some natural $d_{\alpha,i}$. Suppose that
  \begin{equation*}
    \alpha = a_{\alpha,1} [D_1] + \ldots + a_{\alpha,k} [D_k] \text{,}
  \end{equation*}
  where $a_{\alpha,i}$ are integers. Then
  \begin{equation*}
    \deg f = (e_1 + d_{\alpha,1} + a_{\alpha,1}) [D_1] + \ldots + (e_k + d_{\alpha,k} + a_{\alpha,k}) [D_k] \text{.}
  \end{equation*}
  This is true for any $\alpha \in \Cl X_\Sigma$ such that $g_\alpha \neq 0$, so we can set $c_i = d_{\alpha,i} + a_{\alpha,i}$. We know that $c_i
  \geq a_{\alpha,i}$ for each $\alpha$, so $c_i \geq \max_\alpha a_{\alpha,i}$. If $c_i > \max_\alpha a_{\alpha,i}$ for some $i$, then $x_i$ divides
  $\hat{f}$, which is a contradiction. Hence $c_i = \max_\alpha a_{\alpha,i}$, and therefore $d_{\alpha,i} = \max_{\alpha}a_{\alpha,i} -
  a_{\alpha,i}$. It follows that $\hat{f}$ is the homogenization of $\pi^*(f)$.
\end{proof}

\begin{definition}
  Suppose $I \subseteq \mathbb{C}[x_{k + 1},\dots, x_r]$ is a non-zero ideal. Let
  \begin{equation*}
    I^\text{h} = (f^\text{h} | f \in I\setminus \{0\})
  \end{equation*}
  be the homogenization of $I$.
\end{definition}
\begin{proposition}
  \label{proposition:hom_saturated}
  The ideal $I^\text{h}$ is saturated with respect to $x_1 \ccdots x_k$.
\end{proposition}
\begin{proof}
  Suppose that $x_1 \ccdots x_k f \in I^\text{h}$ for some non-zero $f$, then
  \begin{equation*}
    x_1 \ccdots x_k f = g_1 f_1^\text{h} + \ldots + g_l f_l^\text{h}
  \end{equation*}
  for some $g_i \in S$ and $f_i \in I$. If we set $x_1,\dots,x_k$ to $1$, we get
  \begin{equation*}
    \pi^*(f) = \pi^*(g_1) f_1 + \ldots + \pi^*(g_l) f_l\text{.}
  \end{equation*}
  This means that $\pi^*(f) \in I$, and it follows that $(\pi^*(f))^\text{h} \in I^\text{h}$ from the definition of $I^\text{h}$. But $f$ is divisible
  by $(\pi^*(f))^\text{h}$ from Proposition \ref{proposition:hom_dehom}, so $f \in I^\text{h}$.
\end{proof}
\begin{proposition}
  \label{proposition:binomial}
  Suppose $I \subseteq \mathbb{C}[x_{k+1},\dots,x_{r}]$ is generated by binomials of the form $x^{a} - x^{b}$. Then
  \begin{equation*}
    I^\text{h} = ( (x^a - x^b)^\text{h} | x^a - x^b \in I \setminus 0)\text{.}
  \end{equation*}
\end{proposition}
\begin{lemma}
  \label{lemma:binomial}
  Suppose $I \subseteq \mathbb{C}[x_{k+1},\dots,x_{r}]$ is generated by binomials of the form $x^{a} - x^{b}$ and that $f \in I$. Then there are
  binomials $x^{c_i} - x^{d_i} \in I$ and $\lambda_i \in \mathbb{C}$, where $i = 1,2,\dots,l$, such that
  \begin{equation*}
    f = \sum_{i = 1}^l \lambda_i (x^{c_i} - x^{d_i})
  \end{equation*}
  and every $x^{c_i}, x^{d_i}$ appears as a monomial of $f$ with a non-zero coefficient.
\end{lemma}
\begin{proof}
  Suppose
  \begin{equation*}
    f = \sum_{i = 1}^m \kappa_i (x^{a_i} - x^{b_i})\text{,}
  \end{equation*}
  where $x^{a_i} - x^{b_i} \in I$ and $\kappa_i \in \mathbb{C}\setminus\{0\}$ for $i = 1,2,\dots,m$. Suppose that some monomial $x^b$ appears in the sum on
  right-hand side and that it does not appear on the left-hand side. Possibly changing the signs of some $\kappa_i$, we may assume that there are
  indices $i_1,\dots,i_n$ such that $b_{i_1} = \dots = b_{i_n} = b$ and $\sum_{j=1}^n \kappa_{i_j} = 0$, and that $b$ appears nowhere else in the sum on
  the right-hand side. In this case $\kappa_{i_1} = -\sum_{j=2}^n \kappa_{i_j}$ and therefore
  \begin{equation*}
    f = \sum_{i | b_i \neq b}\kappa_i (x^{a_i} - x^{b_i}) + \sum_{j = 2}^{n}\kappa_{i_j}(x^{a_{i_j}} - x^{a_{i_1}}) \text{,}
  \end{equation*}
  where each $x^{a_{i_j}} - x^{a_{i_1}} = (x^{a_{i_j}} - x^{b_{i_j}}) - (x^{a_{i_1}} - x^{b_{i_1}}) \in I$. We have reduced the number of summands on
  the right-hand side. Continuing this process, we get to the situation where every monomial on the right-hand side appears on the left-hand side with
  a non-zero coefficient.
\end{proof}
\begin{proof}[Proof of Proposition \ref{proposition:binomial}]
  Let $f \in I$ be a non-zero polynomial. From Lemma \ref{lemma:binomial} we get that there are $\lambda_i \in \mathbb{C}$ and $x^{c_i} - x^{d_i}$,
  where $i = 1,\dots,l$, such that
  \begin{equation*}
    f = \sum_{i = 1}^l \lambda_i (x^{c_i} - x^{d_i})
  \end{equation*}
  and every $x^{c_i}, x^{d_i}$ appears as a monomial $x^{e_s}$ of $f$ with a non-zero coefficient. Let $s(i), s'(i)$ be such that $x^{c_i} =
  x^{e_{s(i)}}$ and $x^{d_i} = x^{e_{s'(i)}}$. We need to get back to the definition of homogenization. Suppose $x^{e_s}$ has
  class $a_{s,1} [D_1] + \dots + a_{s,k} [D_k]$. Then $x^{c_i}$ has class $a_{s(i),1} [D_1] + \dots + a_{s(i),k} [D_k]$, and $x^{d_i}$ has class
  $a_{s'(i),1} [D_1] + \dots + a_{s'(i),k} [D_k]$. It follows that
  \begin{equation*}
    (x^{c_i} - x^{d_i})^\text{h} = x_1^{b_{i,1} - a_{s(i),1}} \ccdots x_k^{b_{i,k} - a_{s(i),k}} x^{c_i} - x_1^{b_{i,1} - a_{s'(i),1}} \ccdots
    x_k^{b_{i,k} - a_{s'(i),k}} x^{d_i}\text{,}
  \end{equation*}
  where $b_{i,j} = \max(a_{s(i),j}, a_{s'(i),j})$. Let
  \begin{equation*}
    f = \sum_{s=1}^m \mu_s x^{e_s}\text{.}
  \end{equation*}
  Then
  \begin{equation*}
    f^\text{h} = \sum_{s=1}^m \mu_s x_1^{\bar{b}_1 - a_{s,1}} \ccdots x_k^{\bar{b}_k - a_{s,k}} x^{e_s}\text{.}
  \end{equation*}
  Here 
  \begin{multline*}
    \bar{b}_j = \max\{ a_{s,j} | s = 1,\dots,m \} = \max\{ a_{s(i),j}, a_{s'(i),j} | i =1,\dots,l \} \\
    = \max\{ b_{i,j} | i=1,\dots,l
  \}\text{,}
  \end{multline*}
  as every $x^{c_i}, x^{d_i}$ appears as a monomial of $f$.
  Hence
  \begin{align*}
    f^\text{h} &= \sum_{s=1}^l \mu_s x_1^{\bar{b}_1 - a_{s,1}} \ccdots x_k^{\bar{b}_k - a_{s,k}}x^{e_s} \\
    &= \sum_{i=1}^l \lambda_i\Big(x_1^{\bar{b}_1 - a_{s(i),1}} \ccdots x_k^{\bar{b}_k - a_{s(i),k}} x^{c_i} 
    - x_1^{\bar{b}_1 - a_{s'(i),1}} \ccdots x_k^{\bar{b}_k - a_{s'(i),k}} x^{d_i}\Big) \\
    &= \sum_{i=1}^l \lambda_i x_1^{\bar{b}_1- b_{i,1}} \ccdots x_k^{\bar{b}_k - b_{i,k}} \Big(x_1^{b_{i,1}- a_{s(i),1}} \ccdots x_k^{b_{i,k} - a_{s(i),k}}
    x^{c_i}\\ &- x_1^{b_{i,1} - a_{s'(i),1}} \ccdots x_k^{b_{i,k} - a_{s'(i),k}} x^{d_i}\Big) \\
    &= \sum_{i=1}^l \lambda_i x_1^{\bar{b}_1- b_{i,1}} \ccdots x_k^{\bar{b}_k - b_{i,k}}(x^{c_i} - x^{d_i})^\text{h}\text{.}
  \end{align*}
\end{proof}

\subsection{Embedded tangent space}
Let $X_P$ be the toric variety embedded by a very ample polytope $P$ with vertices in lattice $M$. Let $v$ be a vertex of the polytope $P$ (which
corresponds to a torus fixed point $p \in X_P$).
\begin{proposition}\label{proposition:embedded_tangent_space}
  The projective embedded tangent space at $p$ in the embedding by $P$ is given by the linear space of the lattice points of $P$ which belong to the
  Hilbert basis of the semigroup $\mathbb{N}(P\cap M - v)$.
\end{proposition}
\begin{proof}
  Let $z_0,\dots,z_k$ be the coordinates corresponding to the monomials in the embedding by $P$, with $z_0$ corresponding to the vertex $v$. Let us
  look at the affine chart given by setting $z_0 = 1$. The equations of the toric variety in the affine chart come from integral relations between the
  lattice points of $P \cap M - v$. The equations of the embedded tangent space at $p$ (in the affine chart) are given in the following way: the forms
  \begin{equation*}
    \sum_{i=1}^k \frac{\partial f}{\partial z_i}_{|[z_1,\dots,z_k] = [0,0,\dots,0]} z_i \text{,}
  \end{equation*}
  where $f$ is an equation of the embedded variety $X_P$ in the affine chart, give all the equations.
  
  Suppose $h_j$ is an element of the Hilbert basis of $\mathbb{N}(P\cap M - v)$, and that the coordinate $z_j$ corresponds to $h_j$. Any relation is
  of the form
  \begin{equation}\label{relation:of_lattice_points}
    l h_j + l_1 h_{j_1} + \dots + l_m h_{j_m} = l_{m + 1} h_{j_{m + 1}} + \dots + l_n h_{j_n} \text{,}
  \end{equation}
  where $h_{j_i} \in P\cap M - v$, $h_{j_i}$ are mutually different, $h_{j_i} \neq h_j$, and $l \geq 0, l_i \geq 1$ for $i =1,\dots, n$ are positive
  integers. In this relation the left-hand side is not equal to $h_j$ as $h_j$ is not a sum. Let us look at the polynomial equation coming form this
  Relation \eqref{relation:of_lattice_points}. It is
  \begin{equation}\label{equation:of_toric_variety}
    z_j^{l} z_{j_1}^{l_1}\ccdots z_{j_m}^{l_m} = z_{j_{m+1}}^{l_{m+1}} \ccdots z_{j_n}^{l_n} \text{,}
  \end{equation}
  where the left-hand side is not equal to $z_j$, and the right-hand side does not contain $z_j$. If we differentiate this equation with respect to
  $z_j$ and substitute the point p (which has coordinates $(z_1,\dots,z_k) = (0,\dots,0))$, we get $0 = 0$. Therefore $z_j$ does not appear in the
  equation of the embedded tangent space at $p$ coming from Equation \eqref{equation:of_toric_variety}. Hence, the point $(z_1, z_2,\dots,z_k) =
  (0,\dots,0,1,0,\dots,0)$ (where the $1$ is on the $j$-th place) satisfies this equation of embedded tangent space at $p$. 

  Since the point $(0,\dots,0,1,0,\dots,0)$ is independent of the chosen Equation \eqref{equation:of_toric_variety}, we get that it satisfies all the
  equations of the projectivized tangent space at $p$.

  We come back the projective coordinates. We proved that every for every vector $h_j$ in the Hilbert basis of $\mathbb{N}(P \cap M - v)$ the point
  $[0,\dots,0,1,0,\dots,0]$ (where the $1$ is on the $j$-th place) is in the embedded tangent space. Also the point $p = [1,0,\dots,0]$ is in this
  space. As this projective space has dimension equal to the cardinality of the Hilbert basis (see \cite[Lemma 1.3.10]{cox_book}), we get the desired
  equality.
\end{proof}

\section{Apolarity}\label{section:apolarity}
Recall the definition of the apolarity action from Equation \eqref{equation:apolarity}. When $g\apolar F = 0$ we often say that $g$ is apolar to $F$.
The grading on $T$ is the same as on $S$:
$$\deg y_i \coloneqq [D_{\rho_i}]\text{.}$$
\begin{remark}
  Notice that $\apolar$ defined in Equation \eqref{equation:apolarity} could be seen as derivation, except that we do not multiply by a constant.  We
  only need to replace $y_i^b$ with $b!\cdot y_i^b$. This can be done by taking $T$ to be the ring of divided powers, see \cite[Appendix
  A]{iarrobino_kanev_book_Gorenstein_algebras}, or \cite[Chapter A2.4]{eisenbud} for a coordinate free version. For characteristic zero, this amounts
  to setting $y_i^{(b)} = \frac{y_i^b}{b!}$. But here we do not need $T$ to be a ring, we only need it to be a module. So we might as well write
  $y_i^b$ instead of $y_i^{(b)}$. It will not matter, provided we do not multiply $y_i^{b_1}$ by $y_i^{b_2}$. This will make some calculations easier.
\end{remark}
\begin{remark}
  Notice that when we take $g \in S_\alpha$ and $F \in T_\beta$, then $g\apolar F$ is homogeneous of degree $\alpha-\beta$ for any $\alpha$, $\beta
  \in \Cl X_\Sigma$.  That follows from the fact that when we multiply by subsequent $x_i$'s, the degree of $F$ decreases by $[D_{\rho_i}]$. This
  means that, although $T$ is not a graded $S$-module, it becomes a graded $S$-module if we define the grading by
  \begin{equation*}
    \deg y_i = -[D_{\rho_i}]\text{.}
  \end{equation*}
\end{remark}

Futhermore, if $F \in T$ is homogeneous, we will denote by $F^\perp$ its annihilator, which is a homogeneous ideal in that case.

Assume $X_\Sigma$ is a complete toric variety. Then we have $S_0 = T_0 = \mathbb{C}$ and $S_\alpha, T_\alpha$ are finite-dimensional vector spaces for
any $\alpha \in \Cl X_\Sigma$.
\begin{proposition}
  \label{proposition:duality}
  The map $S_{\alpha}\times T_{\alpha}\to T_0 = \mathbb{C}$ given by $(g, F) \mapsto g \apolar F$ makes the
  \begin{equation*}
    \{x_1^{a_1}\ccdots x_r^{a_r}|[a_1D_{\rho_1}+\ldots+a_rD_{\rho_r}]=\alpha\}
  \end{equation*}
  basis dual to
  \begin{equation*}
    \{y_1^{b_1}\ccdots y_r^{b_r}|[b_1D_{\rho_1}+\ldots+b_rD_{\rho_r}]=\alpha\}\text{.}
  \end{equation*}
  for any $\alpha \in \Cl X_\Sigma$. In particular, it is a duality for any $\alpha \in \Cl X_\Sigma$.
\end{proposition}
\begin{proof}
  We know that $x_1^{a_1} \ccdots x_r^{a_r} \apolar y_1^{a_1} \ccdots y_r^{a_r} = 1$. Consider the value of
  $x_1^{a_1} \ccdots x_r^{a_r} \apolar y_1^{b_1} \ccdots y_r^{b_r}$ when $(a_1,\dots,a_r) \neq (b_1,\dots,b_r)$. We know that
  \begin{equation}
    x_1^{a_1} \ccdots x_r^{a_r} \apolar y_1^{b_1} \ccdots y_r^{b_r} =
    \begin{cases}
      y_1^{b_1-a_1}\ccdots y_r^{b_r-a_r} & \text{if} \: b_i \geq a_i \: \text{for all} \: i \text{.} \\
      0 & \text{otherwise.}
    \end{cases}
    \label{equation:result_of_duality}
  \end{equation}
  We want to prove (\ref{equation:result_of_duality}) is zero, so suppose otherwise. The degree of (\ref{equation:result_of_duality}) is zero. But the
  only monomial whose degree is the trivial class is the constant monomial $1$ (we use that $X_\Sigma$ is complete). This implies that $b_i = a_i$ for
  all $i$. But this cannot be true, since we assumed $(a_1,\dots,a_r) \neq (b_1,\dots,b_r)$. This contradiction means that $x_1^{a_1}\ccdots x_r^{a_r}
  \apolar y_1^{b_1}\ccdots y_r^{b_r} = 0$, as desired.
\end{proof}

As a corollary, we see that $T = \bigoplus_{\alpha \in \Cl X_\Sigma}{H^0(X, \mathcal{O}(\alpha))^*}$.

Combining Proposition \ref{proposition:duality} and Corollary \ref{corollary:isomorphism}, we get
\begin{proposition}\label{proposition:formula}
  Let $X_\Sigma$ be a complete toric variety. Then for any $\alpha \in \Pic X_\Sigma$ such that $\mathcal{O}(\alpha)$ is basepoint free, the map
  \begin{equation*}
    \varphi \colon X_\Sigma \to \mathbb{P}(H^0(X_\Sigma, \mathcal{O}(\alpha))^*)
  \end{equation*}
  associated with the complete linear system $|\mathcal{O}(\alpha)|$ is given by
  \begin{equation}\label{equation:formula}
    \varphi([\lambda_1,\dots,\lambda_r]) = \left[\sum_{\substack{b_1,\ldots,b_r \in \mathbb{Z}_{\geq 0}| \\
    y_1^{b_1}\ccdots y_r^{b_r} \in T_\alpha}}\lambda_1^{b_1}\ccdots\lambda_r^{b_r}\cdot y_1^{b_1}\ccdots
    y_r^{b_r}\right]\text{.}
  \end{equation}
\end{proposition}
\begin{proof}
  In general, if $\{s_i | i \in I \}$ is a basis of $H^0(X, \mathcal{O}(\alpha))$ ($I$ is some finite index set), and
  $\{s^i | i\in I \} \subseteq H^0(X, \mathcal{O}(\alpha))^*$ is the dual basis, then
  \begin{equation*}
    \varphi(p) = \left [\sum_{i\in I}s_i(p)\cdot s^i\right ]\text{,}
  \end{equation*}
  where $s_i(p)$ means evaluating section $s_i$ at point $p$. Note that it does not make sense to talk about the value of a section in $\mathbb{C}$,
  but the quotient $s_i(p)/s_j(p) \in \mathbb{C}$ makes sense, and the sum makes sense as a class in the projectivization of $H^0(X,
  \mathcal{O}(\alpha))^*$.

  By Proposition \ref{proposition:duality}, the monomials $y_1^{b_1}\dots y_r^{b_r} \in T_\alpha$ form a dual basis to $x_1^{b_1}\dots x_r^{b_r}$.  So
  from Corollary \ref{corollary:isomorphism} we know that for any $i = (b_1,\dots,b_r)$, $i' = (b_1',\dots,b_r')$ such that $s_{i'}(p)$ is non-zero we
  have
  \begin{equation*}
    \frac{s_i(p)}{s_{i'}(p)} = \frac{(x_1^{b_1}\ccdots x_r^{b_r})(p)}{(x_1^{b_1'}\ccdots x_r^{b_r'})(p)} = \frac{\lambda_1^{b_1}\ccdots
    \lambda_r^{b_r}}{\lambda_1^{b_1'}\ccdots \lambda_r^{b_r'}}\text{.}
  \end{equation*}
  The formula (\ref{equation:formula}) follows.
\end{proof}

\subsection{Hilbert function}\label{subsection:hilbert_function}
Let $\alpha \in \Pic X_\Sigma$ be a very ample class. Fix $F \in T_\alpha$. The ring $S/{F^\perp}$ is called the apolar ring of $F$. It is graded by
the class group of $X_\Sigma$. Let us denote it by $A_F$. Consider its Hilbert function $H : \Cl X_\Sigma \to \mathbb{Z}_{\geq 0}$ given by
\begin{equation*}
  \beta \mapsto dim_\mathbb{C}\left((A_F)_\beta\right)\text{.}
\end{equation*}
The Hilbert function is symmetric. The proof for projective space also applies to toric varieties:
\begin{proposition}\label{proposition:hilbert_function_symmetry}
  Let $X_\Sigma$ be a complete toric variety. Then for any $\beta \in \Cl X_\Sigma$:
  \begin{equation*}
    dim_{\mathbb{C}}(A_F)_\beta = dim_{\mathbb{C}}(A_F)_{\alpha-\beta}\text{.}
  \end{equation*}
\end{proposition}
\begin{proof}
  We will prove that the bilinear map $(A_F)_\beta \times (A_F)_{\alpha-\beta} \to \mathbb{C} \cong (A_F)_0$ given by $(g, h) \mapsto (g\cdot
  h)\apolar F$ is a duality. Take any $g \in S_\beta$ such that $g \apolar F \neq 0$. Then there is $h \in S_{\alpha-\beta}$ such that $h \apolar
  (g\apolar F) \neq 0$ (because $\apolar$ makes $S_{\alpha-\beta}$ and $T_{\alpha-\beta}$ dual by Proposition \ref{proposition:duality}). But this
  means that $(h\cdot g) \apolar F \neq 0$. We have proven that multiplying by any non-zero $g \in (A_F)_\beta$ is non-zero as a map
  $(A_F)_{\alpha-\beta} \to \mathbb{C}$. Similarly, multiplying by any non-zero $h \in (A_F)_{\alpha-\beta}$ is non-zero as a map $(A_F)_{\beta} \to
  \mathbb{C}$. We are done. 
\end{proof}
\begin{remark}\label{remark:catalecticant}
  The values of the Hilbert function of $S/{F^\perp}$ are the same as the ranks of the catalecticant homomorphisms. More precisely, let
  \begin{equation*}
    C_F^\beta : S_\beta \to T_{\alpha-\beta}
  \end{equation*}
  be given by
  \begin{equation*}
    g \mapsto g \apolar F\text{.}
  \end{equation*}
  This map is called the catalecticant homomorphism. We have
  \begin{equation*}
    \rank{C_F^\beta} = \dim_\mathbb{C}(A_F)_\beta\text{.}
  \end{equation*}
  This is because the graded piece of $F^\perp$ of degree $\beta$ is the kernel of $C_F^\beta$. For more on catalecticant homomorphisms, see
  \cite[Section 2]{teitler_geometric_lower_bounds} or \cite[Chapter 1]{iarrobino_kanev_book_Gorenstein_algebras}.
\end{remark}
\subsection{Apolarity Lemma}\label{section:apolarity_lemma}

Suppose $X$ is a projective variety over $\mathbb{C}$. Let $\mathcal{L}$ be a very ample line bundle on $X$, and $\varphi \colon X \to \mathbb{P}(H^0(X,
\mathcal{L})^*)$ the associated morphism. For a closed subscheme $i\colon R \hookrightarrow X$, $\langle R \rangle$ denotes its linear span in
$\mathbb{P}(H^0(X, \mathcal{L})^*)$, and $\mathcal{I}_R$ denotes its ideal sheaf on $X$. Recall that for any line bundle on $X$, the vector subspace
$H^0(X, \mathcal{I}_R \otimes \mathcal{L}) \subseteq H^0(X, \mathcal{L})$ consists of the sections which pull back to zero on $R$.

Let $(\cdot \apolar \cdot): H^0(X, \mathcal{L}) \otimes H^0(X, \mathcal{L})^* \to \mathbb{C}$ denote the natural pairing (this agrees with the
notation introduced in Equation \eqref{equation:apolarity}). Now we are ready to formulate the Apolarity Lemma:
\begin{proposition}[Apolarity Lemma, general version]
  \label{proposition:general_apolarity}
  Let $F \in H^0(X, \mathcal{L})^*$ be a non-zero element. Then for any closed subscheme $i : R \hookrightarrow X$ we have
  \begin{equation*}
    F \in \langle R \rangle \iff H^0(X, \mathcal{I}_R \otimes \mathcal{L}) \apolar F = 0\text{.}
  \end{equation*}
\end{proposition}
\begin{proof}
  Take any $s \in H^0(X, \mathcal{L})$, let $H_s$ be the corresponding hyperplane in $H^0(X, \mathcal{L})^*$. Then,
  \begin{equation*}
    \langle R \rangle \subseteq H_s \iff i^*(s) = 0 \iff s \in H^0(X, \mathcal{I}_R \otimes \mathcal{L}) \text{.}
  \end{equation*}
  Below we identify sections $s \in H^0(X, \mathcal{L})$ with hyperplanes $H_s$ in $H^0(X, \mathcal{L})^*$. Then for any $R$
  \begin{align*}
    F \in \langle R \rangle & \iff \forall_{s\in H^0(X, \mathcal{L})}(\langle R \rangle \subseteq H_s \implies F \in H_s) \\
    & \iff \forall_{s\in H^0(X, \mathcal{L})}(s \in H^0(X, \mathcal{I}_R \otimes \mathcal{L}) \implies F \in H_s) \\
    & \iff \forall_{s\in H^0(X, \mathcal{L})}(s \in H^0(X, \mathcal{I}_R \otimes \mathcal{L}) \implies s \apolar F = 0) \\
    & \iff H^0(X, \mathcal{I}_R \otimes \mathcal{L}) \apolar F = 0 \text{.}
  \end{align*}
\end{proof}

\begin{proof}[Proof of Theorem \ref{theorem:multigraded_apolarity}]
  From Proposition \ref{proposition:general_apolarity} we know that $F \in \langle R \rangle$ if and only if $I(R)_\alpha \subseteq F^\perp_\alpha$. It
  remains to prove that $I(R)_\alpha \subseteq F^\perp_\alpha$ implies $I(R) \subseteq F^\perp$. Suppose $I(R)_\alpha \subseteq F^\perp_\alpha$. Take
  any $g \in I(R)_\beta$ for some $\beta \in \Cl X_\Sigma$. We want to show that $g \apolar F = 0$. We have $S_{\alpha-\beta}\cdot g \subseteq
  I_\alpha$, because $g$ is in the ideal. This means that $(S_{\alpha-\beta} \cdot g) \apolar F = 0$, i.e.\ $S_{\alpha-\beta}\apolar(g\apolar F) = 0$.
  Now, $g \apolar F$ is an element of $T_{\alpha-\beta}$, which is zero when multiplied by anything from $S_{\alpha-\beta}$, which is equal to
  $T_{\alpha-\beta}^*$ by Proposition \ref{proposition:duality}. It follows that $g \apolar F$ is zero.
\end{proof}
\begin{remark}\label{remark:agreeing_ideals}
  By Proposition \ref{proposition:ideal}, we might have taken $\sat{I(R)}$ instead of $I(R)$ in Theorem \ref{theorem:multigraded_apolarity}. By
  \cite[Theorem 3.7]{cox_homogeneous}, we might have taken any $B$-saturated ideal defining $R$.
\end{remark}

\section{Catalecticant bounds}\label{section:catalecticant_bounds}
We prove lower bounds for various kinds of rank (called catalecticant bounds). They help us to calculate these ranks in Section
\ref{section:examples}. See \cite[Section 2]{teitler_geometric_lower_bounds} for a different viewpoint on these types of lower bounds in the cases of
the Veronese variety, Segre-Veronese variety and general varieties. 
\begin{proposition}\label{proposition:zero_dimensional}
  Let $X$ be a complete variety and $R$ be a zero-dimensional subscheme of $X$ with ideal sheaf $\mathcal{I}_R$.  Then for any line bundle $\mathcal{L}$
  \begin{equation*}
    \length{R} \geq h^0(X, \mathcal{L}) - h^0(X, \mathcal{I}_R \otimes \mathcal{L})\text{.}
  \end{equation*}
  Recall that for a zero-dimensional scheme $\length{R}$ is $\dim_{\mathbb{C}}H^0(R, \mathcal{O}_R)$.
\end{proposition}
\begin{proof}
  We have an exact sequence
  \begin{equation*}
    0 \to \mathcal{I}_R \to \mathcal{O}_X \to \mathcal{O}_R \to 0\text{.}
  \end{equation*}
  We tensor it with $\mathcal{L}$:
  \begin{equation*}
    0 \to \mathcal{I}_R\otimes \mathcal{L} \to \mathcal{L} \to \mathcal{L}_{|R} \to 0\text{.}
  \end{equation*}
  After taking global sections (which are left-exact), we get an exact sequence
  \begin{equation*}
    0 \to H^0(X, \mathcal{I}_R\otimes \mathcal{L}) \to H^0(X, \mathcal{L}) \to H^0(R, \mathcal{L}_{|R})\text{.}
  \end{equation*}
  It follows that 
  \begin{equation*}
    h^0(R, \mathcal{L}_{|R}) \geq h^0(X,\mathcal{L}) - h^0(X, \mathcal{I}_R \otimes \mathcal{L})\text{.}
  \end{equation*}
  But on a zero-dimensional scheme, every line bundle trivializes. This means $h^0(R, \mathcal{L}_{|R}) = h^0(R, \mathcal{O}_R)$, which is the length
  of $R$.
\end{proof}

Let $X_\Sigma$ be a projective simplicial toric variety. Let us fix a very ample class $\alpha \in \Pic X_\Sigma$.  Suppose $\beta \in \Cl X_\Sigma$.
The linear map $\apolar \colon S_{\beta}\otimes T_{\alpha} \to T_{\alpha-\beta}$ can be seen as coming from the morphism
\begin{equation*}
  \mathcal{O}(\beta)\otimes \mathcal{O}(\alpha - \beta) \to \mathcal{O}(\alpha)
\end{equation*}
by taking multiplication of global sections:
\begin{equation*}
  H^0(X_\Sigma, \mathcal{O}(\beta)) \otimes H^0(X_\Sigma, \mathcal{O}(\alpha - \beta)) \to H^0(X_\Sigma, \mathcal{O}(\alpha))
\end{equation*}
and rearranging the terms:
\begin{equation*}
  H^0(X_\Sigma, \mathcal{O}(\beta))\otimes H^0(X_\Sigma, \mathcal{O}(\alpha))^* \xra{\apolar} H^0(X_\Sigma, \mathcal{O}(\alpha -\beta))^*\text{.}
\end{equation*}
For any $\gamma \in \Cl X_\Sigma$ the space $H^0(X_\Sigma, \mathcal{O}(\gamma))$ is $S_\gamma$ and we identify $T_\gamma$ with
$H^0(X_\Sigma, \mathcal{O}(\gamma))^*$ by Proposition \ref{proposition:duality}. Notice that if we fix $F \in H^0(X_\Sigma, \mathcal{O}(\alpha))^*$,
then the map above becomes the catalecticant homomorphism 
\begin{equation*}
  C_F^\beta : H^0(X_\Sigma, \mathcal{O}(\beta)) \to H^0(X_\Sigma, \mathcal{O}(\alpha-\beta))^*
\end{equation*}
from Remark \ref{remark:catalecticant}.
% Question 1. Is this really the same? In particular, does not our definition of \apolar mess things up?
% Question 2. Can we replace by \beta \in \Cl here (the only thing different is that the map is not iso)

As a corollary of Proposition \ref{proposition:zero_dimensional} and the Apolarity Lemma (Theorem \ref{theorem:multigraded_apolarity}), we get the
catalecticant bound in the special case of line bundles.
\begin{corollary}[Catalecticant bound for cactus rank]\label{corollary:cactus_catalecticant_bound}
  For any $\beta \in \Pic X_\Sigma$, and any $F \in H^0(X_\Sigma, \mathcal{O}(\alpha))^*$ we have
  \begin{equation*}
     \crr(F) \geq \rank{C_F^\beta}\text{.}
  \end{equation*}
\end{corollary}
\begin{proof}
  Take any zero-dimensional scheme $R\hookrightarrow X_\Sigma$ such that $F \in \langle R \rangle$. Let $I$ be any $B$-saturated ideal defining $R$.
  We have
  \begin{multline*}
    \length R \geq h^0(X_\Sigma, \mathcal{O}(\beta)) - h^0(X_\Sigma, \mathcal{I}_R \otimes \mathcal{O}(\beta)) = \dim_{\mathbb{C}}(S/I)_\beta \\ 
    \geq \dim_{\mathbb{C}}(S/F^\perp)_\beta  = \dim_{\mathbb{C}} \im C_F^\beta \text{,}
  \end{multline*}
  where the first inequality follows from Proposition \ref{proposition:zero_dimensional}, and the second from Theorem
  \ref{theorem:multigraded_apolarity}. We also used that $I(R)$ agrees with any saturated ideal defining $R$ in degrees coming from $\Pic X_\Sigma$,
  see Remark \ref{remark:agreeing_ideals}, and the fact that values of the Hilbert function are ranks of catalecticant homomorphisms (Remark
  \ref{remark:catalecticant}).
\end{proof}
The bound for cactus rank does not hold for classes $\beta \notin \Pic X_\Sigma$. See Subsection \ref{subsection:weighted_projective_plane} for an
example. But the bound does hold for rank and $\beta \in \Cl X_\Sigma$:
\begin{proposition}[Catalecticant bound for rank]\label{proposition:rank_catalecticant_bound}
  For any $\beta \in \Cl X_\Sigma$, and any $F \in H^0(X_\Sigma, \mathcal{O}(\alpha)^*)$ we have
  \begin{equation*}
    \rr(F) \geq \rank{C_F^\beta}\text{.}
  \end{equation*}
\end{proposition}
The following proof is an adaptation of \cite[the ``suprisingly quick proof'' after equation (8)]{teitler_geometric_lower_bounds}.
\begin{proof}
  For any $\gamma \in \Cl X_\Sigma$ and any $(\lambda_1,\dots,\lambda_r) \in \mathbb{C}^r$, define a polynomial in $y_1,\dots,y_r$
  \begin{equation*}
    \psi_\gamma(\lambda_1,\dots,\lambda_r) = \sum_{\substack{ a_1,\dots,a_r \in \mathbb{Z}_{\geq 0}| \\ y_1^{a_1}\ccdots y_r^{a_r} \in
    T_\gamma}}{\lambda_1^{a_1}\ccdots \lambda_r^{a_r}\cdot y_1^{a_1}\ccdots y_r^{a_r}}\text{.}
  \end{equation*}
  First we prove the formula
  \begin{equation*}
    g \apolar \psi_\alpha(\lambda_1,\dots,\lambda_r) = g(\lambda_1,\dots,\lambda_r) \psi_{\alpha-\beta}(\lambda_1,\dots,\lambda_r)
  \end{equation*}
  for any $g \in S_{\beta}$ (here $g(\lambda_1,\dots,\lambda_r)$ means evaluating the polynomial $g$ at the $\lambda_i$'s). The formula is linear in
  $g$, so we may assume $g = x_1^{b_1}\ccdots x_r^{b_r}$.

  Let $P$ be the set of all monomials $y_1^{a_1}\ccdots y_r^{a_r}$ of degree $\alpha$ such that $g \apolar y_1^{a_1}\ccdots y_r^{a_r} \neq 0$ (i.e.\
  $a_i \geq b_i$ for all $i$). Then the map $g \apolar \cdot$ is a bijection from $P$ onto the set of all monomials of degree $\alpha-\beta$ in
  variables $y_1,\dots,y_r$ (injectivity is clear; for surjectivity note that for any $y_1^{a_1'}\ccdots y_r^{a_r'} \in T_{\alpha-\beta}$ the monomial
  $y_1^{a_1' + b_1}\ccdots y_r^{a_r' +b_r} \in T_\alpha$ is what we are looking for). It follows that
  \begin{align*}
    g \apolar \psi_\alpha(\lambda_1,\dots,\lambda_r) &= \sum_{\substack{ a'_1,\dots,a'_r \in \mathbb{Z}_{\geq 0}| \\ y_1^{a'_1}\ccdots y_r^{a'_r}
    \in T_{\alpha-\beta}}}\lambda_1^{a'_1 + b_1}\ccdots \lambda_r^{a'_r + b_r} y_1^{a'_1}\ccdots y_r^{a'_r}\\ &= \lambda_1^{b_1}\ccdots
    \lambda_r^{b_r}\cdot \sum_{\substack{ a'_1,\dots,a'_r \in \mathbb{Z}_{\geq 0}| \\ y_1^{a'_1}\ccdots y_r^{a'_r} \in
    T_{\alpha-\beta}}}\lambda_1^{a'_1}\ccdots \lambda_r^{a'_r} y_1^{a'_1}\ccdots y_r^{a'_r}\\ &=
    g(\lambda_1,\dots,\lambda_r)\psi_{\alpha-\beta}(\lambda_1,\dots,\lambda_r)\text{.}
  \end{align*}
  
  Now we proceed to the proof of the catalecticant bound. Take $F \in H^0(X, \mathcal{O}(\alpha))^*$. Then $\rr(F)$ is the least $l$ such that $F =
  \psi_\alpha(\mb{\lambda}^1) + \dots + \psi_\alpha(\mb{\lambda}^l)$ for some $\mb{\lambda}^1,\dots,\mb{\lambda}^l \in \mathbb{C}^r$ (basically
  because $\psi_\alpha$ agrees with $\varphi_{|\mathcal{O}(\alpha)|}$ from Proposition \ref{proposition:formula}). We want to bound from above the
  dimension of the image of the map $C^\beta_F : S_\beta \to T_{\alpha-\beta}$, $g \mapsto g \apolar F$. But for any $g \in S_\beta$ we have
  \begin{equation*} g
    \apolar F = g \apolar (\psi_\alpha(\mb{\lambda}^1) + \dots + \psi_\alpha(\mb{\lambda}^l)) = g(\mb{\lambda}^1)\cdot
    \psi_{\alpha-\beta}(\mb{\lambda}^1)+\dots+g(\mb{\lambda}^l)\cdot \psi_{\alpha-\beta}(\mb{\lambda}^l)\text{.}
  \end{equation*}
  So for any $g$ in the domain of the map, the image $C_F^\beta(g)$ is in 
  \begin{equation*}
    \langle \psi_{\alpha-\beta}(\mb{\lambda}^1), \dots, \psi_{\alpha-\beta}(\mb{\lambda}^l) \rangle \text{.}
  \end{equation*}
  It follows that the rank of $C_F^\beta$ is at most $l$. 
\end{proof}

\begin{proposition}\label{proposition:closed_set}
  Fix $\beta \in \Cl X_\Sigma$. Then for any $l \in \mathbb{Z}_{+}$ the set of points $F \in \mathbb{P}(H^0(X_\Sigma,\mathcal{O}(\alpha))^*)$ such that
  $\rank(C_F^\beta) \leq l$ is Zariski-closed.
\end{proposition}
\begin{proof}
  Pick a basis of $H^0(X_\Sigma, \mathcal{O}(\beta))$ and a basis of $H^0(X_\Sigma, \mathcal{O}(\alpha - \beta))$. Then $\apolar$ becomes a matrix
  with entries in $H^0(X_\Sigma, \mathcal{O}(\alpha))$. In order to get the rank of the map $C_F^\beta = \cdot \apolar F$, we evaluate the matrix at
  $F \in H^0(X_\Sigma, \mathcal{O}(\alpha))^*$. Hence the set of those $F$'s such that the rank of $\cdot \apolar F$ is at most $l$ is given by the
  vanishing of the $(l+1)$-th minors of the matrix. These minors are polynomials from $\Sym^\bullet H^0(X_\Sigma, \mathcal{O}(\alpha))$. We are done.
\end{proof}
\begin{corollary}[Catalecticant bound for border rank]\label{corollary:border_catalecticant_bound}
  For any $\beta \in \Cl X_\Sigma$ and any $F \in T_\alpha$ we have
  \begin{equation*}
    \brr(F) \geq \rank C_F^\beta\text{.}
  \end{equation*}
\end{corollary}
\begin{proof}
  Let $k = \brr(F)$. From Proposition \ref{proposition:rank_catalecticant_bound} we know that
  \begin{equation*}
    \sigma^0_k(X) = \{[G] \in \mathbb{P}T_\alpha| \rr(G) \leq k\} \subseteq \{[G] \in \mathbb{P}T_\alpha | \rank C_G^\beta \leq k\}\text{.}
  \end{equation*}
  Since the set on the right hand side is closed (Proposition \ref{proposition:closed_set}), we have
  \begin{equation*}
    \sigma_k(X) = \overline{\sigma_k^0(X)} \subseteq \{[G] \in \mathbb{P}T_\alpha | \rank C_G^\beta \leq k\}\text{,}
  \end{equation*}
  and the claim follows.
%  From Proposition \ref{proposition:closed_set} we know that for $l \coloneqq \rank C_F^\beta$ the set of $F' \in \mathbb{P}(H^0(X_\Sigma,
%  \mathcal{O}(\alpha))^*)$ such that $\rank C_{F'}^\beta \geq \rank C_F^\beta$ is Zariski-open. But from Corollary
%  \ref{corollary:cactus_catalecticant_bound} we know that $\rr(F') \geq \rank C_{F'}^\beta$.  Hence for any $F'$ in some Zariski-open neighbourhood of
%  $F$
%  \begin{equation*}
%    \rr(F') \geq \rank C_{F'}^\beta \geq \rank C_F^\beta\text{.}
%  \end{equation*}
%  It follows that $\brr(F) \geq \rank C_F^\beta$.
\end{proof}
\section{Upper bound on cactus rank}\label{section:upper_bound_on_cactus}
In this section, using apolarity we improve the bound for the cactus rank given in \cite{ballico_bernardi_gesmundo_cactus_rank_segre_veronese}. We
generalize the ideas given first in \cite{bernardi_ranestad_cactus_rank_of_cubics} to the multigraded setting.

Recall the maps of dehomogenization $\pi, \pi^*$ and the notion of homogenization denoted by $f^\text{h}$ from Subsection
\ref{subsection:dehomo_and_homo}. We dehomogenize by setting $x_1,\dots,x_k$ to $1$ and by setting every power $y_i^b$ to $1$, where $1 \leq i \leq
k$. Let $F \in H^0(X, \mathcal{O}(\alpha))^*$, and let $f = \pi(F)$.
\begin{proposition}
  Let $G$ be any homogeneous polynomial in $S$. Let $g = \pi^*(G)$. Suppose $g \apolar f = 0$. Then $G \apolar F = 0$.  
\end{proposition}
\begin{proof}
  Let $\tilde{F} = (y_1\ccdots y_k)^D\cdot F$, where $D > \deg G$ and here the degree means degree in $\mathbb{N}_{\geq 0}$ as a non-homogeneous
  polynomial. We will show that $\pi(G \apolar \tilde{F}) = g \apolar f$. By bilinearity of $\: \lrcorner$, we may assume that $\tilde{F}$ and $G$
  are monomials, i.e.\ $G = x_1^{b_1}\ccdots x_r^{b_r}$ and $\tilde{F} = y_1^{a_1}\ccdots y_r^{a_r}$.  We have
  \begin{equation*}
    G \apolar \tilde{F} = \begin{cases}
      y_1^{a_1 - b_1}\ccdots y_r^{a_r - b_r} & \text{if } a_i \geq b_i \text{ for all } i\text{,} \\
      0 & \text{otherwise.}
    \end{cases}
  \end{equation*}
  and
  \begin{equation*}
    g \apolar f = \begin{cases}
      y_{k+1}^{a_{k+1} - b_{k+1}}\ccdots y_r^{a_r - b_r} & \text{if } a_i \geq b_i \text{ for } i \geq k+1\text{,} \\
      0 & \text{otherwise.}
    \end{cases}
  \end{equation*}
  Since for $i\leq k$ we have $a_i \geq D > b_i$, we get that the conditions $a_i \geq b_i$ for all $i$ and $a_i \geq b_i$ for $i \geq k+1$ are
  equivalent. It follows that $\pi(G \apolar \tilde{F}) = g\apolar f$. 

  As $\pi$ is injective (Proposition \ref{proposition:injectivity}), we immediately get that $G \apolar \tilde{F} = 0$. We know that $F =
  (x_1\ccdots x_k)^D \apolar \tilde{F}$, so
  \begin{equation*}
    G \apolar F = G \apolar ((x_1\ccdots x_k)^D \apolar \tilde{F}) = (x_1\ccdots x_k)^D \apolar (G \apolar \tilde{F}) = 0\text{.}
  \end{equation*}
\end{proof}
This means that $(f^\perp)^\text{h} \subseteq F^\perp$ , so by Theorem \ref{theorem:multigraded_apolarity} the ideal $f^\perp$ defines a scheme $R$
such that $F \in \langle R \rangle$ (notice that $(f^\perp)^\text{h}$ is $B$-saturated, as it is $x_1\ccdots x_r$-saturated by Proposition
\ref{proposition:hom_saturated}). Hence, its length gives an upper bound on the cactus rank of $F$. Let us calculate it in the case of the
Segre-Veronese embedding.

Let $\mathbb{P}^{n_1}\times \dots \times \mathbb{P}^{n_k}$ be embedded by the line bundle $\mathcal{O}(d_1,\dots,d_k)$. Here, after we dehomogenize
$F$, we get that $f$ is a polynomial in $n_1$ variables of degree $(1,0,\dots,0)$, $n_2$ variables of degree $(0,1,0,\dots,0)$, \dots, and $n_k$
variables of degree $(0,\dots,0,1)$. We need to bound from above 
\begin{equation*}
  \dim \pi^*(S)/f^\perp\text{.}
\end{equation*}
This is equal to the dimension of the space of all partial derivatives of $f$. We do this just as in \cite[Proof of Theorem
3]{bernardi_ranestad_cactus_rank_of_cubics}. The space of all multihomogeneous polynomials of degree $(e_1,\dots,e_k)$ in $n_1 + n_2 + \dotsb + n_k$
variables has dimension
\begin{equation*}
  \binom{n_1 -1 + e_1}{e_1}\ccdots \binom{n_k -1 + e_k}{e_k}\text{.}
\end{equation*}
Now let $d = d_1 +\dots + d_k$ be the total degree. We bound the space of the partials of total degree at most $\frac{d}{2}$ by the number of
linearly independent variables, by which we differentiate. We bound the space of the other partials by the dimension of the space of all polynomials
of total degree less than $\frac{d}{2}$. Hence,
\begin{align*}
  \crr(F) &\leq \sum_{\substack{(e_1,\dots,e_k) |\\ e_1 + \dots + e_k \leq d/2 }}\binom{n_1 -1+ e_1}{e_1}\ccdots \binom{n_k -1+ e_k}{e_k} \\
  &+\sum_{\substack{(e_1,\dots,e_k) |\\ e_1 + \dots + e_k > d/2 }}\binom{n_1 -1+ d_1 - e_1}{d_1 - e_1}\ccdots \binom{n_k -1 + d_k - e_k}{d_k - e_k}\text{.}
\end{align*}
This is stronger than the bound in \cite{ballico_bernardi_gesmundo_cactus_rank_segre_veronese}, since the authors include all monomials, and we
include monomials of bounded multidegree.
\begin{example}
  Consider the Segre embedding 
  \begin{equation*}
    \underbrace{\mathbb{P}^n \times \dots \times \mathbb{P}^n}_{k \text{ times}} \hookrightarrow \mathbb{P}(\mathbb{C}^{n+1}\otimes \dots \otimes
    \mathbb{C}^{n+1})\text{,}
  \end{equation*}
  which is given by the line bundle $\mathcal{O}(1,\dots,1)$. Then the bound
  gives for any $F$
  \begin{align*}
    \crr(F) &\leq 1 + kn + \binom{k}{2}n^2 +\dots +\binom{k}{ k/2 } n^{ k/2 } \\
                  &+ \binom{k}{ k/2 - 1} n^{ k/2  - 1} + \dots +\binom{k}{2}n^2 + kn + 1 
  \end{align*}
  for $k$ even, and
  \begin{align*}
    \crr(F) &\leq 1 + kn + \binom{k}{2}n^2 +\dots +\binom{k}{\lfloor k/2 \rfloor} n^{\lfloor k/2 \rfloor} \\
                  &+ \binom{k}{\lfloor k/2 \rfloor} n^{\lfloor k/2 \rfloor} + \dots +\binom{k}{2}n^2 + kn + 1 
  \end{align*}
  for $k$ odd.
\end{example}
\begin{remark}
  The bound is sometimes better if we replace the condition $e_1 + \dots e_k \leq \frac{d}{2}$ by the condition $l(e_1,\dots,e_k) \leq
  \frac{l(d_1,\dots,d_k)}{2}$ for some linear form $l$.
  
  For instance, consider $\mathbb{P}^2\times \mathbb{P}^2 \times \mathbb{P}^2$ embedded by the line bundle $\mathcal{O}(3,3,2)$. We get that
  $l(e_1,e_2,e_3) = e_1 + e_2+ e_3$ gives the bound
  \begin{equation*}
    \crr(F) \leq 255
  \end{equation*}
  for any $F$, while the form $l(e_1,e_2,e_3) = 2e_1 + 2e_2 +3e_3$ gives the bound
  \begin{equation*}
    \crr(F) \leq 250
  \end{equation*}
  for any $F$, and the bound from the article \cite{ballico_bernardi_gesmundo_cactus_rank_segre_veronese} gives 
  \begin{equation*}
    \crr(F) \leq 294
  \end{equation*}
  for any $F$.
\end{remark}

\section{Examples}\label{section:examples}

We use what we proved to look at some examples. In this section, we denote the coordinates of the ring $S$ by Greek letters $\alpha,
\beta,\dots$ and the corresponding coordinates in $T$ by $x, y,\dots$ (possibly with subscripts).

We calculate ranks, cactus ranks and border ranks (denoted by $\rr(F)$, $\crr(F)$, $\brr(F)$) of some monomials  $F$ for toric surfaces embedded into
projective spaces. See Definitions \ref{definition:sigma_rank_border_rank} and \ref{definition:cactus_rank} for the definitions of these ranks.
\subsection{$\mathbb{P}^1\times\mathbb{P}^1$}\label{subsection:p1timesp1}
Consider the set
\begin{equation*}
  \{\rho_{\alpha,0} = (1, 0), \rho_{\alpha,1} = (-1,0), \rho_{\beta,0} = (0, 1), \rho_{\beta,1} = (0, -1)\}\text{.}
\end{equation*}
Let $\Sigma$ be the only complete fan such that this set is its set of rays. Then $X_\Sigma$ is $\mathbb{P}^1 \times \mathbb{P}^1$, which is smooth.

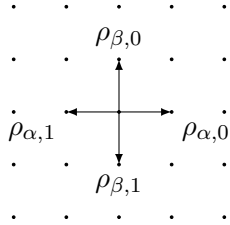
\[\begin{tikzpicture}[scale = 0.7]
  \draw[-latex, thin] (0,0) -- (1,0);
  \draw[-latex, thin] (0,0) -- (0,1);
  \draw[-latex, thin] (0,0) -- (-1,0);
  \draw[-latex, thin] (0,0) -- (0,-1);
  \foreach \x in {-2,...,2}
  {
    \foreach \y in {-2,...,2}
    {
      \draw[fill] (\x,\y)circle [radius=0.025];
    }
  }
  \node[below right] at (1,0) {$\rho_{\alpha,0}$};
  \node[above] at (0,1) {$\rho_{\beta,0}$};
  \node[below left] at (-1,0) {$\rho_{\alpha,1}$};
  \node[below] at (0,-1) {$\rho_{\beta,1}$};
\end{tikzpicture}\]
Its class group is the free abelian group on two generators $D_{\rho_{\alpha,0}} \sim D_{\rho_{\alpha,1}}$ and $D_{\rho_{\beta,0}} \sim
D_{\rho_{\beta,1}}$. Here and later in this section $D_\rho$ is the toric invariant divisor corresponding to $\rho$ (as in Section
\ref{section:toric_varieties}) and $\sim$ means the linear equivalence. Let $\alpha_0, \alpha_1, \beta_0, \beta_1$ be the variables corresponding to
$\rho_{\alpha,0}$, $\rho_{\alpha,1}$, $\rho_{\beta,0}$, $\rho_{\beta,1}$. As a result, we may think of $S$ as the polynomial ring
$\mathbb{C}[\alpha_0, \alpha_1, \beta_0, \beta_1]$ graded by $\mathbb{Z}^2$, where the grading is given by
\begin{center}
  \begin{tabular}{c|cccc}
    $f$ & $\alpha_0$ & $\alpha_1$ & $\beta_0$ & $\beta_1$ \\
    \hline
    \multirow{2}{*}{$\deg f$} & 1 & 1 & 0 & 0  \\
     & 0 & 0 & 1 & 1 
  \end{tabular}
\end{center}
The nef cone in $(\Cl X_\Sigma)_{\mathbb{R}}$ is generated by $D_{\rho_{\alpha,0}}$ and $D_{\rho_{\beta,0}}$. 

Let $x_0, x_1, y_0, y_1$ be the basis dual to $\alpha_0, \alpha_1, \beta_0, \beta_1$. We consider the problem of determining cactus ranks and ranks of
monomials $F = x_0^{k_0}x_1^{k_1}y_0^{l_0}y_1^{l_1}$, where $k_0 \geq k_1 \geq 1, l_0 \geq l_1 \geq 1$. The annihilator ideal is $(\alpha_0^{k_0 + 1},
\alpha_1^{k_1 + 1}, \beta_0^{l_0 + 1}, \beta_1^{l_1 + 1})$. We have
\begin{equation*}
  \dim (S/F^{\perp})_{(k_1, l_1)} = (k_1 + 1)(l_1 + 1)\text{.}
\end{equation*}
It follows that 
\begin{equation*}
  \crr(F) \geq (k_1 + 1)(l_1 + 1)\text{.}
\end{equation*}
But $I = (\alpha_1^{k_1 + 1}, \beta_1^{l_1 + 1}) \subseteq F^\perp$ is a $B$-saturated ideal of a scheme of length $(k_1 + 1)(l_1 + 1)$. This is
because we can look locally, at the affine open set where $\alpha_0, \beta_0 \neq 0$. There our scheme becomes
\begin{equation*}
  \Spec \mathbb{C}\left[\frac{\alpha_1}{\alpha_0},\frac{\beta_1}{\beta_0}\right]
  /\left(\frac{\alpha_1^{k_1 + 1}}{\alpha_0^{k_1 + 1}}, \frac{\beta_1^{l_1 + 1}}{\beta_0^{l_1 + 1}}\right) \cong
  \Spec \mathbb{C}[u,v]/(u^{k_1 + 1}, v^{l_1 + 1})
\end{equation*}
for some variables $u,v$. The scheme constructed in this way has desired length. Hence, by Theorem \ref{theorem:multigraded_apolarity}
\begin{equation*}
  \crr(F) = (k_1 + 1)(l_1 + 1)\text{.}
\end{equation*}

Now we address the problem of finding ranks of such monomials. We prove Theorem \ref{theorem:inequalities}.

Let $R$ be the polynomial ring $\mathbb{C}[u,v]$.
\begin{lemma}
  \label{lemma:bezout}
  Consider the ideal $I = (u^m v^n - 1, u^p - v^q) \subseteq R$, where $m, n\geq 1$ and at least one of the integers $p,q$ is greater than or equal to
  $1$. Then $V(I) \subset \mathbb{A}^2$ consists of $mq + np$ reduced points.
\end{lemma}
\begin{proof}
  First we show that $u^m v^n - 1, u^p - v^q$ intersect transversally. Let $s \in \mathbb{N}_0$ be the smallest number such that $u^s v^t - 1 \in
  I$ for some $t \in \mathbb{Z}_+$. Let $i \in \mathbb{N}_0$ be the smallest number such that $u^i  - v^j \in I$ for some $j \in \mathbb{Z}_+$. We
  claim that $s = 0$ or $i = 0$. Assume to the contrary, that $\min(s,i) > 0$. If $s \geq i$, then we have
  \begin{equation*}
    u^s v^t - 1 - u^{s - i}v^t(u^i - v^j) = u^{s-i}v^{j + t} -1 \in I\text{,}
  \end{equation*}
  which contradicts the minimality of $s$. If $s < i$, then
  \begin{equation*}
    v^t(u^i - v^j) - u^{i-s}(u^s v^t - 1) = u^{i-s} - v^{t + j} \in I \text{,}
  \end{equation*}
  which contradicts the minimality of $i$.

  We get that $v^i - 1 \in I$ for some $i \in \mathbb{Z}_+$. Similarly, by interchanging the roles of $i,s$ with $j, t$, we get that $u^j -1 \in
  I$ for some $j \in \mathbb{Z}_+$. The polynomials $v^i - 1$ and $u^j - 1$ intersect transversally in $ij$ points, so $u^m v^n -1, u^p - v^q$ also
  intersect transversally.

  We want to use B\'ezout's theorem for $\mathbb{P}^1\times \mathbb{P}^1$. In order to do so, we homogenize generators of $I$ and check that they have
  no roots at infinity. We consider the dehomogenization given by the ring homomorphism $S \to R$, $\alpha_0 \mapsto u$, $\alpha_1 \mapsto 1$,
  $\beta_0 \mapsto v$, $\beta_1 \mapsto 1$. Then the generators of $I$ become
  \begin{equation}
    \label{equation:generators}
    \alpha_0^m\beta_0^n - \alpha_1^m \beta_1^n, \alpha_0^p \beta_1^q - \alpha_1^p \beta_0^q\text{.}
  \end{equation}
  Now we can see that if $\alpha_1 = 0$, then $\alpha_0 \neq 0$, so if we put this into the first generator in Equation \ref{equation:generators}, we
  get that $\beta_0 = 0$, and if we put it into the second generator, we get $\beta_1 = 0$. But $\beta_0$ and $\beta_1$ cannot simultaneously be $0$.
  Similarly, if $\beta_1 = 0$, then $\beta_0 \neq 0$. From the first generator, we get that $\alpha_0 = 0$, and from the second we have $\alpha_1 =
  0$, but the two equalities cannot hold at the same time.

  This means that the polynomials $u^m v^n - 1, u^p - v^q$ have no common roots at infinity, so we can use multihomogeneous B\'ezout's theorem (see
  \cite[Example 4.9]{MR1328833}) to get that $u^m v^n - 1, u^p -v^q$ have $mq + np$ common roots.
\end{proof}

\begin{proof}[Proof of Item \eqref{item:first_inequality} of Theorem \ref{theorem:inequalities}]
  If $k_0 = k_1$ or $l_0 = l_1$, Item \eqref{item:first_inequality} becomes Equation \eqref{equation:obvious_rank_inequality}, so it is true. Now
  assume $k_0 > k_1$ and $l_0 > l_1$ and consider $I = (u^{k_1 + 1} - v^{l_1 + 1}, u^{k_0 + 1}v^{l_0 - l_1} - 1)$. Let 
  \begin{equation*}
    M = (k_0 + 1)(l_1 + 1) + (k_1+ 1)(l_0 + 1) - (k_1 + 1)(l_1 + 1)\text{.}
  \end{equation*}
  From Proposition \ref{proposition:hom_dehom} we know that $I^\text{h}$ is $\alpha_1 \beta_1$-saturated, which implies that it is $B$-saturated (we
  homogenize in the same way as in the proof of Lemma \ref{lemma:bezout}). By Lemma \ref{lemma:bezout}, we get that $I^\text{h}$ is a radical ideal of
  $M$ points. We need to show that $I^\text{h} \subseteq F^\perp$. From Proposition \ref{proposition:binomial} it suffices to show that for $u^m v^n -
  1 \in I$ we have $(u^m v^n - 1)^\text{h} \in F^\perp$ and that for $u^p - v^q \in I$ we have $(u^p - v^q)^\text{h} \in F^\perp$. 

  Since $u^{k_1 + 1} - v^{l_1 + 1} \in I$, from Lemma \ref{lemma:bezout} we get that any element of the form $u^m v^n - 1 \in I$ must satisfy $(k_1 + 1)n +
  (l_1 + 1)m \geq M$.

  Let us picture the polynomials on $\mathbb{Z}^2$. We put the polynomial $u^m v^n - 1$ in the point $(m,n)$ (that is the degree of the homogenization).
  Then the binomials $u^m v^n - 1 \in I$ lie above or on the line connecting two points: $(k_0 + 1, l_0 - l_1)$ and $(k_0 - k_1, l_0 + 1)$. Similarly,
  the binomials $u^p - v^q \in I$ lie above or on the line $(k_0 + 1)q + (l_0 - l_1)p = M$.

  We have:

  \emph{Claim 1:} for each binomial of the form $u^m v^n - 1 \in I$ we have either $m \geq k_0 + 1$ or $n \geq l_0 + 1$. It suffices to argue that there
  are no elements $u^m v^n - 1 \in I$ in the interior of the segment connecting points $(k_0 + 1, l_0 - l_1)$ and $(k_0 - k_1, l_0 + 1)$ nor in the
  interior of the triangle with vertices $(k_0 +1,l_0 - l_1), (k_0 - k_1, l_0 + 1), (k_0 + 1, l_0 + 1)$. Suppose that $u^m v^n - 1$ is such, then
  \begin{equation*}
    u^m v^n - 1 - (u^{k_0 + 1} v^{l_0 - l_1} - 1) = u^m v^{l_0 - l_1} (v^{n - l_0 + l_1} - u^{k_0 + 1 - m}) \in I\text{.}
  \end{equation*}
  Since $I$ is $uv$-saturated, we get that $u^{k_0 + 1 - m} - v^{n - l_0 + l_1} \in I$. But $k_0 + 1 - m < k_1 + 1$ and $n - l_0 + l_1 < l_1 + 1$
  (because $m > k_0 - k_1$ and $n < l_0 + 1$, respectively), so the point $(p,q) = (k_0 + 1 - m, n - l_0 +l_1)$ lies below the line $ (k_0 + 1)q + (l_0
  - l_1)p = M$, a contradiction.

  \emph{Claim 2:} there are no binomials $u^p - v^q \in I$ lying in the interior of the rectangle with vertices $(k_1 + 1, 0), (k_0 + 1,0), (k_1 + 1,l_1
  + 1), (k_0 + 1, l_1+ 1)$. Indeed, for any such binomial $u^p - v^q$ we would have
  \begin{equation*}
    u^p - v^q - u^{p-k_1 -1} (u^{k_1 + 1} - v^{l_1 + 1}) = v^q(u^{p-k_1 - 1} v^{l_1 + 1 - q} - 1) \in I\text{,}
  \end{equation*}
  so also $u^{p - k_1 -1}v^{l_1 + 1 - q} - 1 \in I$. But $p - k_1 - 1 < k_0 - k_1$ and $l_1 + 1 -q < l_0 + 1$ (since $p < k_0 + 1$ and $q > 0$,
  respectively), so the point $(m,n) = (p - k_1 - 1, l_1 + 1 - q)$ lies below the line $(k_1 + 1)n + (l_1 + 1)m = M$, a contradiction.

  \emph{Claim 3:} there are no binomials $u^p - v^q \in I$ lying in the interior of the rectangle with vertices $(0, l_1 + 1), (0,l_0 + 1), (k_1 + 1,l_1
  + 1), (k_1 + 1, l_0+ 1)$. This is just \emph{Claim 2} with the roles of the axes reversed.

  From \emph{Claim 1, 2} and \emph{3} it follows that for each $u^m v^n -1 \in I$ we have $(u^m v^n -1)^\text{h} \in F^\perp$ and for $u^p - v^q \in I$
  we have $(u^p - v^q)^\text{h} \in F^\perp$. From it we conclude that for any binomial $b \in I$ we have $b^\text{h} \in F^\perp$ (since homogenization
  is well-behaved with respect to multiplication by monomials), so from Proposition \ref{proposition:binomial} we have $I^\text{h} \subseteq F^\perp$,
  and we are done with Item \eqref{item:first_inequality}.
\end{proof}
\begin{proof}[Proof of Item \eqref{item:second_inequality}]
  Suppose that $\rr(F) < (k_0 + 1)(l_1 + 1)$. Then by Theorem \ref{theorem:multigraded_apolarity} there is a radical $B$-saturated ideal $I$ of at
  most $(k_0 + 1)(l_1 + 1) - 1$ points such that $I \subseteq F^{\perp} = (\alpha_0^{k_0 + 1}, \alpha_1^{k_1 + 1}, \beta_0^{l_0 + 1}, \beta_1^{l_1 +
  1})$. By Proposition \ref{proposition:zero_dimensional} we have that $\dim (S/I)_{(k_0,l_1)} \leq (k_0 + 1)(l_1+1) -1$. We know that $\dim
  S_{(k_0,l_1)} = (k_0 + 1)(l_1 + 1)$. But this means that $\dim I_{(k_0,l_1)} \geq 1$. We have
  \begin{align*}
    &F^{\perp}_{(k_0,l_1)} = \\
    &\alpha_1^{k_1 + 1} \cdot \langle  \alpha_0^{k_0-k_1 - 1}, \alpha_0^{k_0-k_1 - 2}\alpha_1 ,\dots, \alpha_1^{k_0-k_1 -1} \rangle
    \cdot \langle \beta_0^{l_1}, \beta_0^{l_1 -1} \beta_1,\dots,\beta_1^{l_1}\rangle
  \end{align*}
  Hence there is a non-zero polynomial
  \begin{equation}\label{equation:rank_induction}
    \alpha_1^{k_1 + 1}(\eta_{k_0-k_1 -1}\alpha_0^{k_0-k_1-1} + \eta_{k_0 - k_1 -2}\alpha_0^{k_0 - k_1 -2}\alpha_1 + \ldots + \eta_0\alpha_1^{k_0-k_1-1})
    \in I\text{,}
  \end{equation}
  where $\eta_i \in \langle \beta_0^{l_1}, \beta_0^{l_1-1} \beta_1,\ldots, \beta_1^{l_1} \rangle$.

  We prove by descending induction on $j$ that we have
  \begin{equation*}
    \eta_{k_0 - k_1 - 1} = \eta_{k_0 - k_1 - 2} = \dots = \eta_{j+1} = 0\text{.}
  \end{equation*}
  The beginning of the induction is trivial ($j = k_0 - k_1 -1$). Now assume that the induction assumption
  holds for a given $j$. Then (by Equation \eqref{equation:rank_induction}) for some $l \geq 1$
  \begin{equation*}
    \alpha_1^l(\eta_{j}\alpha_0^j + \eta_{j-1}\alpha_0^{j-1}\alpha_1 + \ldots + \eta_0\alpha_1^j) \in I\text{.}
  \end{equation*}
  The ideal $I$ is radical, so we know that
  \begin{equation*}
    \alpha_1(\eta_{j}\alpha_0^j + \eta_{j-1}\alpha_0^{j-1}\alpha_1 + \ldots + \eta_0\alpha_1^j) \in I\text{.}
  \end{equation*}
  But $I \subseteq F^{\perp}$, so
  \begin{equation*}
    \alpha_1(\eta_{j}\alpha_0^{j}+\eta_{j-1}\alpha_0^{j-1}\alpha_1 + \ldots + \eta_0\alpha_1^{j}) \apolar F = 0\text{.}
  \end{equation*}
  We know that
  \begin{align*}
    \alpha_1(\eta_{j}\alpha_0^j + \ldots + \eta_0\alpha_1^j) \apolar F =& \; \bar{\eta}_j x_0^{k_0 - j} x_1^{k_1 - 1} + \bar{\eta}_{j-1} x_0^{k_0 -j
    + 1} x_1^{k_1 - 2} + \ldots + \\ + &\; \bar\eta_{\max(0,j-k_1 + 1)} x_0^{k_0 - \max(0,j-k_1 + 1)}x_1^{\max(0,k_1 - 1 - j)} \text{,}
  \end{align*}
  where $\bar{\eta}_i = \eta_i \apolar y_0^{l_0} y_1^{l_1} \in \mathbb{C}[y_0, y_1]_{l_0}$ are such that for every $i$ if $\eta_i \neq 0$, then
  $\bar{\eta}_i \neq 0$.  As all the monomials in the sum are different, we get that all $\bar{\eta}_i$ in the sum are $0$, which implies that all
  corresponding $\eta_i$ are $0$. At least the first summand is present in the sum since $k_0 -j \geq k_1 + 1$ and $k_1 - 1 \geq 0$. This is our
  induction assumption for some $j' < j$.

  The fact that $\eta_i$ are all zero gives a contradition with the fact that the polynomial was non-zero. This proves Item
  \eqref{item:second_inequality}.
\end{proof}

\begin{proof}[Proof of Item \eqref{item:third_inequality}]
  Let $I \subseteq F^\perp$ be a
  $B$-saturated radical ideal of at most $(k_1 + 2)(l_1 + 2) - 2$ points. Then $\dim (S/I)_{(k_1 + 1,l_1 + 1)} \leq (k_1 + 2)(l_1 + 2) - 2$, so $\dim
  I_{(k_1 + 1,l_1 + 1)} \geq 2$. Since
  \begin{align*}
    F^\perp_{k_1 + 1, l_1 + 1} = & \alpha_1^{k_1 + 1} \langle \beta_0^{l_1 + 1}, \beta_0^{l_1}\beta_1,\dots,\beta_0 \beta_1^{l_1} \rangle\\
    & \langle \alpha_0^{k_1 + 1}, \alpha_0^{k_1}\alpha_1,\dots,\alpha_0 \alpha_1^{k_1}\rangle \cdot \beta_1^{l_1 + 1} + \langle \alpha_1^{k_1 +
    1}\beta_1^{l_1 + 1} \rangle\text{,}
  \end{align*}
  we get that $I_{(k_1 + 1, l_1 + 1)}$ has a basis consisting of
  \begin{align*}
    t_1 &= \alpha_1^{k_1 + 1}(\kappa_1\beta_0^{l_1}\beta_1 +\dots + \kappa_{l_1}\beta_0\beta_1^{l_1})  \\ 
    &+ (\lambda_0 \alpha_0^{k_1 + 1} + \dots + \lambda_{k_1} \alpha_0 \alpha_1^{k_1})\beta_1^{l_1 + 1} + \eta \alpha_1^{k_1 + 1}\beta_1^{l_1 + 1}\text{,} \\
    t_2 &= \alpha_1^{k_1 + 1}(\mu_0 \beta_0^{l_1 + 1} + \mu_1\beta_0^{l_1}\beta_1 +\dots + \mu_{l_1}\beta_0\beta_1^{l_1})  \\
    &+ (\nu_0 \alpha_0^{k_1 + 1} + \dots + \nu_{k_1}\alpha_0 \alpha_1^{k_1})\beta_1^{l_1 + 1} +  \zeta \alpha_1^{k_1 + 1}\beta_1^{l_1 + 1}\text{,} 
  \end{align*}
  where $\kappa_i, \lambda_i, \mu_i, \nu_i, \eta, \zeta \in \mathbb{C}$.  If $\kappa_1 = 0$, then (since $I$ is radical and $t_1$ is divisible by
  $\beta_1^2$) $t_1/\beta_1 \in I$. Out of the monomials of $t_1/\beta_1$, only the ones divisible by $\alpha_1^{k_1 + 1}$ are in $F^\perp$. Hence,
  $\lambda_i = 0$ for all $i$. But then from the fact that $I$ is radical and that $t_1/\beta_1$ is divisible by $\alpha_1^2$ we get that
  $t_1/(\alpha_1\beta_1) \in I$. Since none of the monomials of $t_1/(\alpha_1\beta_1)$ are in $F^\perp$, we get $\eta = \kappa_2 = \kappa_3 = \dots =
  \kappa_{l_1} = 0$, a contradiction. It follows that we may assume that $\kappa_1 = 1$.

  If $\mu_0 = 0$, then we consider the element $\mu_1 t_1 - t_2$. It is divisible by $\beta_1^2$, so in the same way as before, we get that all the
  coefficients of $\mu_1 t_1 - t_2$ are $0$. It follows that we may assume that $\mu_0 = 1$.

  In this case
  \begin{align*}
    &\beta_0 t_1 - \beta_1 t_2 \\
    &= \alpha_1^{k_1 + 1}((\kappa_2 - \mu_1)\beta_0^{l_1}\beta_1^2 +\dots+ 
    (\kappa_{l_1} - \mu_{l_1 - 1})\beta_0^2 \beta_1^{l_1} -\mu_{l_1}\beta_0 \beta_1^{l_1 + 1}) \\
    &+(\lambda_0 \alpha_0^{k_1 + 1} + \lambda_1 \alpha_0^{k_1}\alpha_1 +\dots \lambda_{k_1}\alpha_0\alpha_1^{k_1})\beta_0\beta_1^{l_1 + 1} + \eta
    \alpha_1^{k_1 + 1} \beta_0\beta_1^{l_1 + 1} \\
    &-(\nu_0 \alpha_0^{k_1 + 1} + \nu_1 \alpha_0^{k_1}\alpha_1 +\dots \nu_{k_1}\alpha_0\alpha_1^{k_1})\beta_1^{l_1 + 2} + \zeta
    \alpha_1^{k_1 + 1} \beta_1^{l_1 + 2} \text{.}
  \end{align*}
  This is divisible by $\beta_1^2$, so also
  \begin{equation*}
    \frac{\beta_0 t_1 - \beta_1 t_2} {\beta_1} \in I\text{.} 
  \end{equation*}
  The monomials $\alpha_0^{k_1 + 1}\beta_0\beta_1^{l_1},\dots,\alpha_0\alpha_1^{k_1}\beta_0 \beta_1^{l_1}$ are not in $F^\perp$ (which is a monomial
  ideal), so $\lambda_i = 0$ for all $i$. Hence $t_1$ is divisible by $\alpha_1^{k_1 + 1}$, and therefore $t_1/\alpha_1^{k_1} \in I\subseteq F^\perp$.
  The monomial $\alpha_1 \beta_0^{l_1} \beta_1$ is not in $F^\perp$, but its coefficient in $t_1/\alpha_1^{k_1}$ is $\kappa_1 = 1$, a contradiction.
\end{proof}

\subsection{Hirzebruch surface $\mathbb{F}_1$}\label{subsection:hirzebruch_surface}
Consider the set
\begin{equation*}
  \{\rho_{\alpha,0} = (1, 0), \rho_{\alpha,1} = (-1,-1), \rho_{\beta,0} = (0, 1), \rho_{\beta,1} = (0, -1)\}\text{.}
\end{equation*}
Let $\Sigma$ be the only complete fan such that this set is the set of rays of $\Sigma$. The example in \cite[Example 3.1.16]{cox_book} is the same,
only with a different ray arrangement. Then $X_\Sigma$ is called the Hirzebruch surface $\mathbb{F}_1$. It is smooth.

\[\begin{tikzpicture}[scale = 0.7]
  \draw[-latex, thin] (0,0) -- (1,0);
  \draw[-latex, thin] (0,0) -- (0,1);
  \draw[-latex, thin] (0,0) -- (0,-1);
  \draw[-latex, thin] (0,0) -- (-1,-1);
  \foreach \x in {-2,...,2}
  {
    \foreach \y in {-2,...,2}
    {
      \draw[fill] (\x,\y)circle [radius=0.025];
    }
  }
  \node[below right] at (1,0) {$\rho_{\alpha,0}$};
  \node[above] at (0,1) {$\rho_{\beta,0}$};
  \node[below left] at (-1,-1) {$\rho_{\alpha,1}$};
  \node[below] at (0,-1) {$\rho_{\beta,1}$};
\end{tikzpicture}\]
Its class group is the free abelian group on two generators $D_{\rho_{\alpha,0}} \sim D_{\rho_{\alpha,1}}$ and $D_{\rho_{\beta,1}}$. Moreover,
$D_{\rho_{\beta,0}} \sim D_{\rho_{\beta,1}} + D_{\rho_{\alpha, 0}}$. Let $\alpha_0, \alpha_1, \beta_0, \beta_1$ be the variables corresponding to
$\rho_{\alpha,0}$, $\rho_{\alpha,1}$, $\rho_{\beta,0}$, $\rho_{\beta,1}$. As a result, we may think of $S$ as the polynomial ring
$\mathbb{C}[\alpha_0, \alpha_1, \beta_0, \beta_1]$ graded by $\mathbb{Z}^2$, where the grading is given by
\begin{center}
  \begin{tabular}{c|cccc}
    $f$ & $\alpha_0$ & $\alpha_1$ & $\beta_0$ & $\beta_1$ \\
    \hline
    \multirow{2}{*}{$\deg f$} & 1 & 1 & 1 & 0  \\
     & 0 & 0 & 1 & 1 
  \end{tabular}
\end{center}
The nef cone in $(\Cl X_\Sigma)_{\mathbb{R}}$ is generated by $D_{\rho_{\alpha,0}}$ and $D_{\rho_{\beta,0}} \sim D_{\rho_{\alpha,0}} +
D_{\rho_{\beta, 0}}$. 

\begin{example}
  Consider the monomial $F \coloneqq x_0 x_1 y_0 y_1$, where $x_0, x_1, y_0, y_1$ is the basis dual to $\alpha_0, \alpha_1, \beta_0, \beta_1$. It has
  degree $(3, 2)$, so it is in the interior of the nef cone, so the corresponding line bundle is very ample. We claim that the rank and the cactus
  rank of $F$ are four, and that the border rank is three:
    \begin{center}
      \begin{tabular}{|c|c|c|}
	\hline
	 $\rr(F)$ & $\crr(F)$ & $\brr(F)$ \\
	\hline
	$4$ & $ 4$ & $3$  \\
	\hline
      \end{tabular}
    \end{center}
    
  Let us compute the Hilbert function of the apolar algebra of $F$.   
  \[\begin{tikzpicture}
    \draw[help lines] (-2,-1) grid (5,5);
    \draw[->, thick] (0,0) -- (0,5);
    \draw[->, thick] (0,0) -- (5,0);
    \draw[fill] (0,0)circle [radius=0.025];
    \draw[fill] (1,0)circle [radius=0.025];
    \draw[fill] (2,0)circle [radius=0.025];
    \draw[fill] (3,0)circle [radius=0.025];
    \draw[fill] (0,1)circle [radius=0.025];
    \draw[fill] (1,1)circle [radius=0.025];
    \draw[fill] (2,1)circle [radius=0.025];
    \draw[fill] (3,1)circle [radius=0.025];
    \draw[fill] (0,2)circle [radius=0.025];
    \draw[fill] (1,2)circle [radius=0.025];
    \draw[fill] (2,2)circle [radius=0.025];
    \draw[fill] (3,2)circle [radius=0.025];
    \node[below left] at (0,0) {1};
    \node[below left] at (1,0) {2};
    \node[below left] at (2,0) {1};
    \node[below left] at (3,0) {0};
    \node[below left] at (0,1) {1};
    \node[below left] at (1,1) {3};
    \node[below left] at (2,1) {3};
    \node[below left] at (3,1) {1};
    \node[below left] at (0,2) {0};
    \node[below left] at (1,2) {1};
    \node[below left] at (2,2) {2};
    \node[below left] at (3,2) {1};
  \end{tikzpicture}\]
  Notice that it can only be non-zero in the first quadrant. Hence, the symmetry of the Hilbert function (see Proposition
  \ref{proposition:hilbert_function_symmetry}) implies it can only be non-zero in the rectangle with vertices $(0,0)$, $(3,0)$, $(3,2)$, $(0,2)$.
  Computing each value of the Hilbert function is just computing the kernel of a linear map. For instance, for degree $(1,0)$, we have
  \begin{equation*}
    (a_0 \alpha_0 + a_1 \alpha_1) \apolar x_0 x_1 y_0 y_1 = a_0 x_1 y_0 y_1 + a_1 x_0 y_0 y_1\text{,}
  \end{equation*}
  which is zero if and only if $a_0 = 0$ and $a_1 = 0$. Hence, the Hilbert function is
  \begin{equation*}
    dim_{\mathbb{C}} (S/F^\perp)_{(1, 0)} = dim_{\mathbb{C}}S_{(1,0)} - dim_{\mathbb{C}}F^\perp_{(1,0)} = 2 - 0 = 2\text{.}
  \end{equation*}
  For degree $(2, 1)$, we get
  \begin{multline*}
    (a \alpha_0^2 \beta_1 + b \alpha_0\alpha_1\beta_1 + c \alpha_1^2\beta_1 + d\alpha_0\beta_0 + e\alpha_1\beta_0) \apolar x_0 x_1 y_0 y_1 
    = b y_0 + d x_1 y_1 + e x_0 y_1\text{.}
  \end{multline*}
  So the result is zero precisely for vectors of the form $(a,0,c,0,0)$, where $a, c \in \mathbb{C}$. Then
  \begin{equation*}
    dim_{\mathbb{C}} (S/F^\perp)_{(2, 1)} = 5 - 2 = 3\text{.}
  \end{equation*}
  
  The apolar ideal $F^\perp$ is $(\alpha_0^2, \alpha_1^2, \beta_0^2, \beta_1^2)$. (It is independent of the grading, so we can just copy the result
  from the Waring rank case, see \cite{ranestad_schreyer_on_the_rank_of_a_symmetric_form}.)

  Firstly, we will show that the rank is at most four. By the Apolarity Lemma, toric version (Theorem \ref{theorem:multigraded_apolarity}), it is
  enough to find a reduced zero-dimensional subscheme of length four $R$ of $X_\Sigma$ (i.e.\ a set of four points in $X_\Sigma$) such that $I(R)
  \subseteq F^\perp$.  The subscheme defined by $I = (\alpha_0^2 -\alpha_1^2, \beta_0^2 - \alpha_1^2 \beta_1^2) \subseteq F^\perp$ satisfies these
  requirements. This scheme is a reduced union of four points: $[1, 1; 1, 1], [1, 1; 1, -1], [1, -1; 1, 1], [1, -1; 1, -1]$. As a consequence, we may
  write
  \begin{equation*}
    x_0 x_1 y_0 y_1 = \frac{1}{4}\left(\varphi(1, 1; 1, 1) - \varphi(1, 1; 1, -1) - \varphi(1, -1; 1, 1) + \varphi(1, -1; 1, -1)\right)\text{.}
  \end{equation*}

  We will show that the cactus rank is at least four. Suppose it is at most three. Then there is a $B$-saturated homogeneous ideal $I \subseteq F^\perp$
  defining a zero-dimensional subscheme $R$ of length at most three. From the calculation of the Hilbert function, we know that
  $\dim_{\mathbb{C}}{F^\perp_{(2,1)}} = 2$. Let us calculate $\dim_{\mathbb{C}} I_{(2,1)}$. Since $I$ is $B$-saturated, by Proposition
  \ref{proposition:ideal}, the vector subspace $I_{(2,1)} \subseteq S_{(2,1)}$ are the sections which are zero on $R$.  But from Proposition
  \ref{proposition:zero_dimensional}
  \begin{equation*}
    3 \geq \text{length of} \, R \geq \dim_{\mathbb{C}}S_{(2,1)} - \dim_{\mathbb{C}}I_{(2,1)} = 5 -\dim_{\mathbb{C}}I_{(2,1)}\text{,}
  \end{equation*}
  so
  \begin{equation*}
    \dim_{\mathbb{C}} I_{(2,1)} \geq 2\text{.}
  \end{equation*}
  By the Apolarity Lemma (Theorem \ref{theorem:multigraded_apolarity}), we have $I_{(2,1)} \subseteq (F^\perp)_{(2,1)}$. As the dimensions are equal,
  it follows that $I_{(2,1)} = (F^\perp)_{(2,1)}$. This means $\alpha_0^2 \beta_1, \alpha_1^2 \beta_1 \in I$. But $I$ is $B$-saturated, so $\alpha_0
  \alpha_1 \beta_1 \in I \subseteq F^\perp$, which implies that $\alpha_0 \alpha_1 \beta_1 \apolar x_0 x_1 y_0 y_1 = 0$, a contradiction. 

  Let us show that border rank of $F$ is at most three. Take $p = [\lambda, 1; 1, \mu] \in \mathbb{F}_1$. Then from Proposition
  \ref{proposition:formula}, we know that \begin{align*}
    [\lambda, 1; 1, \mu] \mapsto \lambda \mu\cdot \bigg( & \lambda^2 \mu x_0^3 y_1^2 + \lambda \mu x_0^2 x_1 y_1^2 + \mu x_0 x_1^2 y_1^2 + 
    \frac{\mu}{\lambda} x_1^3 y_1^2 \\
    + {} & \lambda x_0^2 y_0 y_1 + x_0 x_1 y_0 y_1 + \frac{1}{\lambda} x_1^2 y_0 y_1 \\
    + {} & \frac{1}{\mu} x_0 y_0^2 + \frac{1}{\mu \lambda} x_1 y_0^2\bigg )    \text{.}  
  \end{align*}
  But
  \begin{equation*}
    [0, 1; 1, \mu] \mapsto \mu\cdot \bigg( {\mu} x_1^3 y_1^2 + x_1^2 y_0 y_1 + \frac{1}{\mu} x_1 y_0^2 \bigg )  \text{,}  
  \end{equation*}
  and
  \begin{equation*}
    [1, 0; 1, 0] \mapsto x_0 y_0^2 \text{.}
  \end{equation*}
  Hence,
  \begin{multline*}
    - x_0 x_1 y_0 y_1 + \frac{1}{\lambda \mu} \varphi([\lambda, 1; 1, \mu]) - \frac{1}{\lambda \mu} \varphi([0, 1; 1, \mu]) -
    \frac{1}{\mu}\varphi([1, 0; 1, 0]) \\
    = \lambda^2 \mu x_0^3 y_1^2 + \lambda \mu x_0^2 x_1 y_1^2 + \mu x_0 x_1^2 y_1^2 + \lambda x_0 y_0^2 \xrightarrow{\lambda, \mu \to 0} 0 \text{.}
  \end{multline*}
  It follows that $x_0 x_1 y_0 y_1$ is expressible as a limit of linear combinations of three points on $X_\Sigma$, so the border rank is at most three.

  But there is another proof that the border rank of $F$ is at most three. We will show that the third secant variety $\sigma_3(X) = \mathbb{P}^8$. It
  suffices to show that $\dim \sigma_3(X_\Sigma)$ is eight. We will use Terracini's Lemma (Proposition \ref{proposition:terracini}).

  Since $X_\Sigma \to \mathbb{P}(H^0(X_\Sigma, \mathcal{O}(\alpha))^*)$ is given by a parametrization, we can calculate the projectivized tangent
  space. Take points of the form $[1, \lambda; \mu, 1]$, where $\lambda, \mu \in \mathbb{C}$. Then
  \begin{equation*}
    \varphi([1, \lambda; \mu, 1]) = [1, \lambda, \lambda^2, \lambda^3, \mu, \mu \lambda, \mu \lambda^2, \mu^2, \mu^2 \lambda]\text{.}
  \end{equation*}
  The coordinates are in the standard monomial basis of $H^0(X_\Sigma, \mathcal{O}(\alpha))^*$. The affine tangent space at $\varphi([1, \lambda; \mu,
  1])$ is spanned by the vector 
  \begin{equation*}
    v = [1, \lambda, \lambda^2, \lambda^3, \mu, \mu \lambda, \mu \lambda^2, \mu^2, \mu^2 \lambda]
  \end{equation*}
  and its two derivatives with respect to $\lambda$ and $\mu$:
  \begin{align*}
    \frac{\partial v}{\partial \lambda} & = [0, 1, 2 \lambda, 3 \lambda^2, 0, \mu, 2\mu \lambda, 0, \mu^2]\text{,} \\
    \frac{\partial v}{\partial \mu} & = [0, 0, 0, 0, 1, \lambda, \lambda^2, 2\mu, 2\mu \lambda]\text{,} \\
  \end{align*}
  If we take three general points, say $[1, x, y, 1], [1, s, t, 1], [1, u, v, 1]$, we can look at the space spanned by the three tangent
  spaces. This will be the space spanned by the rows of the following matrix:
  \begin{equation*}
    M = \left(
  \begin{matrix}
    1 & x & x^2 & x^3 & y & y x & y x^2 & y^2 & y^2 x \\
    0 & 1 & 2 x & 3 x^2 & 0 & y & 2y x & 0 & y^2 \\
    0 & 0 & 0 & 0 & 1 & x & x^2 & 2y & 2y x \\
    1 & s & s^2 & s^3 & t & t s & t s^2 & t^2 & t^2 s \\
    0 & 1 & 2 s & 3 s^2 & 0 & t & 2t s & 0 & t^2 \\
    0 & 0 & 0 & 0 & 1 & s & s^2 & 2t & 2t s \\
    1 & u & u^2 & u^3 & v & v u & v u^2 & v^2 & v^2 u \\
    0 & 1 & 2 u & 3 u^2 & 0 & v & 2v u & 0 & v^2 \\
    0 & 0 & 0 & 0 & 1 & u & u^2 & 2v & 2v u \\
  \end{matrix}
  \right)
  \end{equation*}
  We can calculate the determinant using for instance Macaulay2
  \begin{equation*}
    \det M = (s - u)(u - x)(s - x)(ys - xt - yu + tu + xv - sv)^4\text{.}
  \end{equation*}
  This is non-zero for general points on the variety. This means that the tangent space of the cone of the third secant variety at a general point has
  dimension nine, so $\dim \sigma_3(X_\Sigma) = 8$, hence $\sigma_3(X)$ fills the whole space.

  Finally, the border rank is at least three by Corollary \ref{corollary:border_catalecticant_bound}. We use it for the class $(2,1)$, recall
  that $\dim_{\mathbb{C}}(S/F^\perp)_\beta = \rank C_F^\beta$.

  \begin{remark}\label{remark:wild_case}
    We could also define the smoothable $X$-rank:
    \begin{equation*}
      \srr_X(F) = \min \{\length R | R \hookrightarrow X, \dim R = 0, F \in \langle R \rangle, R \text{ smoothable} \}\text{.} 
    \end{equation*}
    For the definition of a smoothable scheme, see \cite[Definition 5.16]{iarrobino_kanev_book_Gorenstein_algebras}. For more on the smoothable rank,
    see \cite{nisiabu_jabu_smoothable_rank_example}. We always have $\crr(F) \leq \srr(F) \leq \rr(F)$, so in the case of $\mathbb{F}_1$ and $F = x_0
    x_1 y_0 y_1$ we get $\srr(F) = 4$. In particular, we obtain what the authors in \cite{nisiabu_jabu_smoothable_rank_example} call a ``wild'' case,
    i.e.\ the border rank is strictly less than the smoothable rank.
  \end{remark}
\end{example}

\begin{example}
  For a similar case on the same variety, let $F = x_0^2 x_1^2 y_0 y_1$, then $\deg F = (5,2)$. Here the line bundle $\mathcal{O}(5,2)$ gives an
  embedding of $X_\Sigma$ into $\mathbb{P}^{14}$. We show that here the rank and the cactus rank are six, and that the border rank is five:
  \begin{center}
    \begin{tabular}{|c|c|c|}
      \hline
       $\rr(F)$ & $\crr(F)$ & $\brr(F)$ \\
      \hline
      $6$ & $ 6$ & $5$  \\
      \hline
    \end{tabular}
  \end{center}
  The apolar ideal is $F^\perp = (\alpha_0^3, \alpha_0^3, \beta_0^2, \beta_1^2)$.  The Hilbert function of $S/F^\perp$ is the following:

  \[\begin{tikzpicture}
    \draw[help lines] (-2,-1) grid (7,5);
    \draw[->, thick] (0,0) -- (0,5);
    \draw[->, thick] (0,0) -- (7,0);
    \foreach \x in {0,1,2,3,4,5}
    {
      \foreach \y in {0,1,2}
      {
        \draw[fill] (\x,\y)circle [radius=0.025];
      }
    }
    \node[below left] at (0,0) {1};
    \node[below left] at (1,0) {2};
    \node[below left] at (2,0) {3};
    \node[below left] at (3,0) {2};
    \node[below left] at (4,0) {1};
    \node[below left] at (5,0) {0};
    \node[below left] at (0,1) {1};
    \node[below left] at (1,1) {3};
    \node[below left] at (2,1) {5};
    \node[below left] at (3,1) {5};
    \node[below left] at (4,1) {3};
    \node[below left] at (5,1) {1};
    \node[below left] at (0,2) {0};
    \node[below left] at (1,2) {1};
    \node[below left] at (2,2) {2};
    \node[below left] at (3,2) {3};
    \node[below left] at (4,2) {2};
    \node[below left] at (5,2) {1};
  \end{tikzpicture}\]
  The ideal $I = (\alpha_0^3 - \alpha_1^3, \beta_0^2 - \beta_1^2 \alpha_1^2) \subseteq F^\perp$ is a $B$-saturated radical homogeneous ideal defining
  a subscheme of length six, so the rank is at most six.

  Suppose there is a homogeneous $B$-saturated ideal $I \subseteq F^\perp$ defining a subscheme of length five. We have
  \begin{align*}
    S_{(3,1)} &= \langle \alpha_0^2 \beta_0, \alpha_0 \alpha_1 \beta_0, \alpha_1^2 \beta_0, \alpha_0^3 \beta_1, \alpha_0^2 \alpha_1 \beta_1,
    \alpha_0 \alpha_1^2 \beta_1, \alpha_1^3 \beta_1 \rangle\text{, and} \\
    (F^\perp)_{(3, 1)} &= \langle \alpha_0^3 \beta_1, \alpha_1^3 \beta_1 \rangle\text{.}
  \end{align*}
  From Propostion \ref{proposition:zero_dimensional} we have $\dim_{\mathbb{C}}(S/I)_{(3,1)} \leq 5$, so $\dim_{\mathbb{C}}I_{(3,1)} \geq 7 - 5 = 2$.
  But $I_{(3,1)} \subseteq (F^\perp)_{(3,1)}$ from the Apolarity Lemma (Theorem \ref{theorem:multigraded_apolarity}), and also $\dim_{\mathbb{C}}
  (F^\perp)_{(3,1)} = 2$. This means that $I_{(3,1)} = (F^\perp)_{(3,1)}$. 

  Hence, $\alpha_0^3 \beta_1, \alpha_1^3 \beta_1 \in I$. As $I$ is $B$-saturated, we get $\alpha_0^2 \alpha_1^2 \beta_1 \in I \subseteq F^\perp$, but
  this is a contradiction since $\alpha_0^2 \alpha_1^2 \beta_1 \apolar F \neq 0$.

  The border rank is at least five because of Corollary \ref{corollary:border_catalecticant_bound}. Similarly to what we did before, we show that
  fifth secant variety fills the whole space, so the border rank of any polynomial is at most five. Here $\varphi = \varphi_{|\mathcal{O}(5,2)|}$ is
  given (in the standard monomial basis) by
  \begin{equation*}
    [1, \lambda; \mu, 1] \mapsto [1, \lambda, \lambda^2, \lambda^3, \lambda^4, \lambda^5, \mu, \lambda \mu, \lambda^2 \mu, \lambda^3 \mu, \lambda^4
    \mu, \mu^2, \lambda \mu^2, \lambda^2 \mu^2, \lambda^3 \mu^2]\text{.}
  \end{equation*}
  The tangent space of the affine cone of $X_\Sigma$ is spanned by $v = \varphi(1, \lambda; \mu, 1)$ and the two derivatives 
  \begin{align*}
    \frac{\partial v}{\partial \lambda} & = [0, 1, 2 \lambda, 3 \lambda^2, 4 \lambda^3, 5 \lambda^4,
      0, \mu, 2\lambda \mu, 3 \lambda^2 \mu, 4\lambda^3 \mu, 0, \mu^2, 2 \lambda \mu^2, 3\lambda^2 \mu^2]\text{,} \\
    \frac{\partial v}{\partial \mu} & = [0, 0, 0, 0,0,0, 1, \lambda, \lambda^2, \lambda^3, \lambda^4, 2\mu, 2\lambda \mu, 2\lambda^2 \mu, 2\lambda^3
    \mu] \text{.} \\
  \end{align*}
  If we take five points, say $[1, x; y, 1], [1, u; v, 1], [1, s; t, 1], [1, a, b, 1], [1, c, d, 1]$, we get that the tangent space of the affine cone
  of $\sigma_5(X_\Sigma)$ is spanned by the rows of the following matrix:
  \begin{equation*}
    \begin{pmatrix}
    1& x& x^2& x^3& x^4& x^5& y& x y& x^2 y& x^3 y& x^4 y& y^2& x y^2& x^2 y^2& x^3 y^2 \\
    0& 1& 2 x& 3 x^2& 4 x^3& 5 x^4& 0& y& 2x y& 3 x^2 y& 4x^3 y& 0& y^2& 2 x y^2& 3x^2 y^2\\
    0& 0& 0& 0 & 0 & 0 & 1& x& x^2& x^3& x^4& 2y& 2x y& 2x^2 y& 2x^3 y\\
    1& s& s^2& s^3& s^4& s^5& t& s t& s^2 t& s^3 t& s^4 t& t^2& s t^2& s^2 t^2& s^3 t^2 \\
    0& 1& 2 s& 3 s^2& 4 s^3& 5 s^4& 0& t& 2s t& 3 s^2 t& 4s^3 t& 0& t^2& 2 s t^2& 3s^2 t^2\\
    0& 0& 0& 0 & 0 & 0 & 1& s& s^2& s^3& s^4& 2t& 2s t& 2s^2 t& 2s^3 t\\
    1& u& u^2& u^3& u^4& u^5& v& u v& u^2 v& u^3 v& u^4 v& v^2& u v^2& u^2 v^2& u^3 v^2 \\
    0& 1& 2 u& 3 u^2& 4 u^3& 5 u^4& 0& v& 2u v& 3 u^2 v& 4u^3 v& 0& v^2& 2 u v^2& 3u^2 v^2\\
    0& 0& 0& 0 & 0 & 0 & 1& u& u^2& u^3& u^4& 2v& 2u v& 2u^2 v& 2u^3 v\\
    1& a& a^2& a^3& a^4& a^5& b& a b& a^2 b& a^3 b& a^4 b& b^2& a b^2& a^2 b^2& a^3 b^2 \\
    0& 1& 2 a& 3 a^2& 4 a^3& 5 a^4& 0& b& 2a b& 3 a^2 b& 4a^3 b& 0& b^2& 2 a b^2& 3a^2 b^2\\
    0& 0& 0& 0 & 0 & 0 & 1& a& a^2& a^3& a^4& 2b& 2a b& 2a^2 b& 2a^3 b\\
    1& c& c^2& c^3& c^4& c^5& d& c d& c^2 d& c^3 d& c^4 d& d^2& c d^2& c^2 d^2& c^3 d^2 \\
    0& 1& 2 c& 3 c^2& 4 c^3& 5 c^4& 0& d& 2c d& 3 c^2 d& 4c^3 d& 0& d^2& 2 c d^2& 3c^2 d^2\\
    0& 0& 0& 0 & 0 & 0 & 1& c& c^2& c^3& c^4& 2d& 2c d& 2c^2 d& 2c^3 d\\
  \end{pmatrix}
  \end{equation*}
  If we set $(x, y, s, t, u, v, a, b, c, d) = (1,2,3,4,5,6,7,9,0,2)$ and calculate the determinant in the field $\mathbb{Z}/101$, we get $34$, in
  particular, non-zero. This means that the determinant calculated in $\mathbb{C}$ is also non-zero at this point, so it is non-zero on a dense open
  subset. Hence by Terracini's lemma (Proposition \ref{proposition:terracini}) the dimension of the affine cone of $\sigma_5(X_\Sigma)$ is fifteen. It
  follows that $\sigma_5(X_\Sigma) = \mathbb{P}^{14}$. Thus the border rank of $F$ is five.
\end{example}

\subsection{Weighted projective plane $\mathbb{P}(1,1,4)$}\label{subsection:weighted_projective_plane}
Consider a set of rays $\{\rho_x = (-1,-4), \rho_y = (1,0), \rho_z = (0,1)\}$. Let $\Sigma$ be the complete fan determined by these rays. This is a
fan of $\mathbb{P}(1,1,4)$, the weighted projective space with weights $1,1,4$, see \cite[Section 2.0, Subsection Weighted Projective Space; and
Example 3.1.17]{cox_book}.
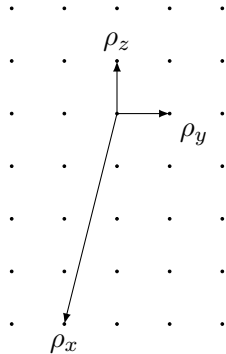
  \[\begin{tikzpicture}[scale = 0.7]
    \draw[-latex, thin] (0,0) -- (-1,-4);
    \draw[-latex, thin] (0,0) -- (1,0);
    \draw[-latex, thin] (0,0) -- (0,1);
    \foreach \x in {-2,...,2}
    {
      \foreach \y in {-4,...,2} 
      {
	\draw[fill] (\x,\y)circle [radius=0.025];
      }
    }
  \node[below right] at (1,0) {$\rho_{y}$};
  \node[above] at (0,1) {$\rho_{z}$};
  \node[below] at (-1,-4) {$\rho_{x}$};
  \end{tikzpicture}\]
  The class group is $\mathbb{Z}$, generated by $D_{\rho_x} \sim D_{\rho_y}$, and we know that $D_{\rho_z} \sim 4 D_{\rho_x}$.  The Cox ring is
  $\mathbb{C}[\alpha, \beta, \gamma]$, where $\alpha, \beta, \gamma$ correspond to $\rho_x, \rho_y, \rho_z$, and the degrees are given by the vector
  $(1,1,4)$. Let $x, y, z$ denote the dual coordinates. The Picard group is generated by $\mathcal{O}(4)$. The only singular point is $[0,0,1]$.

  Consider the embedding given by $\mathcal{O}_{X_\Sigma}(4)$, which is a line bundle. It maps $X$ into $\mathbb{P}^5$ (since there are six monomials
  of degree $4$: $x^4, x^3 y, x^2 y^2, xy^3, y^4, z$). We calculate various ranks of $F = x^2 y^2$. The results are shown it the following table

    \begin{center}
      \begin{tabular}{|c|c|c|}
	\hline
	 $\rr(F)$ & $\crr(F)$ & $\brr(F)$ \\
	\hline
	$3$ & $2$ & $3$  \\
	\hline
      \end{tabular}
    \end{center}

  The Hilbert function of $A_F$ is $(1,2,3,2,1)$ (here the first element of the sequence corresponds to $\mathcal{O}_{X_\Sigma}$, the next to
  $\mathcal{O}_{X_\Sigma}(1)$, and so on). This means (by Corollary \ref{corollary:border_catalecticant_bound}) that $\brr(F) \geq 3$.

  We know that $F^\perp = (\alpha^3, \beta^3, \gamma)$, since the annihilator remains the same if we change the grading. Let $I = (\alpha^3, \beta^3)
  \subseteq F^\perp$. We show that the length of the scheme $R \coloneqq V(I)$ is two. This will mean that $\crr(F) \leq 2$. Since $R$ is supported at
  the point $[0,0,1]$, we can look at it on the affine open $U_\sigma$, where $\sigma = \Cone(\rho_x, \rho_y)$. After localizing $S =
  \mathbb{C}[\alpha,\beta,\gamma]$ at $\gamma$ and taking degree $0$, we get the ring
  \begin{equation*}
    \mathbb{C}\left[\frac{\alpha^4}{\gamma}, \frac{\alpha^3\beta}{\gamma}, \frac{\alpha^2\beta^2}{\gamma}, \frac{\alpha\beta^3}{\gamma},
    \frac{\beta^4}{\gamma}\right]\text{.}
  \end{equation*}
  Ideal $I$ becomes the ideal generated by $\frac{\alpha^4}{\gamma}, \frac{\alpha^3\beta}{\gamma}, \frac{\alpha\beta^3}{\gamma},
  \frac{\beta^4}{\gamma}$ in this ring, so the quotient is a two-dimensional vector space with basis $1, \frac{\alpha^2 \beta^2}{\gamma}$. Hence the
  length of $R$ is two.

  But the cactus rank cannot be $1$, since $x^2 y^2$ is not in the image of $\varphi_{|\mathcal{O}(4)|}$ (see Proposition \ref{proposition:formula}).
  It follows that $\crr(F) = 2$.

  Now consider the ideal $I = (\alpha^3 - \beta^3, \gamma) \subseteq F^\perp$. We show that the length of the scheme defined by $I$ is three. Since
  $I$ is radical, the scheme given by $I$ is reduced, hence this will show that $\rr(F) \leq 3$, as desired. But $I = (\alpha - \beta, \gamma) \cap
  (\alpha - \varepsilon \beta, \gamma)\cap (\alpha -\varepsilon^2 \beta,\gamma)$, where $\varepsilon = \frac{-1 + \sqrt{3}i}{2}$, so the scheme given by $I$
  is the reduced union of $[1,1,0], [\varepsilon,1,0], [\varepsilon^2,1,0]$.

  \begin{remark}\label{remark:catalecticant_counterexample}
    Since in this example 
    \begin{equation*}
      \rank C_F^{\mathcal{O}(2)} = \dim (A_F)_2 = \dim (S/F^\perp)_2 = 3\text{,}
    \end{equation*}
    and $\crr(F) = 2$, we see that the bound stated in point (1) of Theorem \ref{theorem:catalecticant} does not hold for the cactus rank (and
    reflexive sheaves of rank one that are not line bundles).
  \end{remark}
  \begin{remark}
    One can also calculate that the projective tangent space in this embedding at the singular point $[0,0,1]$ is the whole $\mathbb{P}^5$ (this is a
    straightforward application of Proposition \ref{proposition:embedded_tangent_space}). It follows that the cactus rank of every point in
    $\mathbb{P}^5$ is at most two, since any point of the tangent space at $[0,0,1]$ can be reached by a linear span of a scheme of length two
    supported at $[0,0,1]$.
  \end{remark}
  \subsection{Fake weighted projective plane}\label{subsection:fake_weighted_projective_plane}
Consider the set of rays $\{\rho_0 = (-1, -1), \rho_1 = (2, -1), \rho_2 = (-1, 2) \}$. Let $\Sigma$ be the complete fan determined by these rays. Then
$X_\Sigma$ is an example of a fake weighted projective space, see \cite[Example 6.2]{wero_fps}. 

\[\begin{tikzpicture}[scale = 0.7]
  \draw[-latex, thin] (0,0) -- (-1,-1);
  \draw[-latex, thin] (0,0) -- (2,-1);
  \draw[-latex, thin] (0,0) -- (-1,2);
  \foreach \x in {-2,...,2}
  {
    \foreach \y in {-2,...,2} 
    {
      \draw[fill] (\x,\y)circle [radius=0.025];
    }
  }
  \node[below left] at (-1,-1) {$\rho_{0}$};
  \node[right] at (2,-1) {$\rho_{1}$};
  \node[above] at (-1,2) {$\rho_{2}$};
\end{tikzpicture}\]
Let $\alpha_0, \alpha_1, \alpha_2$ be the corresponding coordinates in $S$. The class group is generated by $D_{\rho_0}, D_{\rho_1}, D_{\rho_2}$
with relations $D_{\rho_0} \sim 2 D_{\rho_1} - D_{\rho_2} \sim 2 D_{\rho_2} - D_{\rho_1}$. This is the same as a group with two generators
$D_{\rho_0}$ and $D_{\rho_2} - D_{\rho_1}$ with the relation $3(D_{\rho_2} - D_{\rho_1}) \sim 0$. This choice gives an isomorphism with $\mathbb{Z} \times
\mathbb{Z}/3$ sending $D_{\rho_0}$ to $(1, 0)$ and $D_{\rho_2} - D_{\rho_1}$ to $(0, 1)$.  The Picard group is the subgroup generated by $3
D_{\rho_0}$. It is free.

As a result, $S = \mathbb{C}[\alpha_0, \alpha_1, \alpha_2]$ is graded by $\Cl X_\Sigma = \mathbb{Z}\times \mathbb{Z}/3$, where
\begin{align*}
  \deg \alpha_0 &= (1,0)\text{,} \\
  \deg \alpha_1 &= (1,1)\text{,} \\ 
  \deg \alpha_2 &= (1,-1) = (1, 2)\text{,}
\end{align*}
and $\Pic X_\Sigma$ is generated by $(3,0)$. The singular points of $X_\Sigma$ are $[1,0,0]$, $[0,1,0]$, $[0,0,1]$.

Consider the line bundle $\mathcal{O}(6,0)$. It is ample, because by \cite[Proposition 6.3.25]{cox_book} every complete toric surface is projective,
and the line bundles $\mathcal{O}(-3m, 0)$ for $m < 0$ have no non-zero sections. By \cite[Proposition 6.1.10, (b)]{cox_book} it is very ample. It
gives an embedding $\varphi :X_\Sigma \hookrightarrow \mathbb{P}^9$. We denote the dual coordinates by $x_0, x_1, x_2$.

\begin{example}
  Let $F = x_0^4 x_1 x_2 \in H^0(X_\Sigma, \mathcal{O}(6,0))^*$. The apolar ideal is $(\alpha_0^5, \alpha_1^2, \alpha_2^2)$. We claim that the cactus
  rank is two, the rank is at most five, and the border rank is two.
  \begin{center}
    \begin{tabular}{|c|c|c|}
      \hline
       $\rr(F)$ & $\crr(F)$ & $\brr(F)$ \\
      \hline
      $\leq 5$ & $ 2$ & $2$  \\
      \hline
    \end{tabular}
  \end{center}
  Note that $F$ is not in the image of $\varphi_{|\mathcal{O}(6,0)|}$, so the cactus rank and the border rank are at least two.

  We show that the cactus rank is two. Consider the ideal $I = (\alpha_1^2, \alpha_2^2) \subseteq F^\perp$. It is saturated, since $B$ in this case is
  $(\alpha_0, \alpha_1, \alpha_2)$, so it is the same as in the case of $\mathbb{P}^2$. We show that the length of the subscheme given by $I$ is two.
  Since the support of the scheme is the point $[1, 0, 0]$, we check it on the set $U_\sigma$, where $\sigma = \Cone(\rho_1, \rho_2)$. We localize
  with respect to $\alpha_0$, take degree zero, and get the ring 
  \begin{equation}\label{equation:length_calculation}
    \mathbb{C}\left[\frac{\alpha_1^3}{\alpha_0^3}, \frac{\alpha_2^3}{\alpha_0^3}, \frac{\alpha_1 \alpha_2}{\alpha_0^2}\right] 
    \cong \mathbb{C}[u, v, w]/(w^3 - uv)\text{.}
  \end{equation}

  If we factor out by the ideal generated by $\alpha_1^2$ and $\alpha_2^2$, we get
  \begin{equation*}
    \mathbb{C}[u,v, w]/(w^3 - uv, u, v, w^2) \cong \mathbb{C}[w]/(w^2)\text{,}
  \end{equation*}
  so the length of the scheme defined by $I$ is two.
  
  Now we show that the rank is at most five. Take a homogeneous ideal $I = (\alpha_0^5 - \alpha_1^4 \alpha_2, \alpha_1^3 - \alpha_2^3)
  \subseteq F^\perp$. We show that the length of the subscheme defined by $I$ is five. From these equations we know that no coordinate can be
  zero, so we can check the length on the open subset $U_\sigma$, where $\sigma = \Cone(\rho_1, \rho_2)$. We get the same ring as in Equation
  \ref{equation:length_calculation}, and we want to factor it out by the ideal generated by $\alpha_0^5 - \alpha_1^4 \alpha_2$ and $\alpha_1^3 -
  \alpha_2^3$. The second generator gives the relation $u - v$, and the first one the relation $1 - v w$. So we get the ring
  \begin{equation*}
    \mathbb{C}[v, w]/(w^3 - v^2, 1 - v w)\text{.}
  \end{equation*}
  But notice that $1 = v w$ implies that $w$ is non-zero. Hence
  \begin{multline*}
    \mathbb{C}[v, w]/(w^3 - v^2, 1 - v w) \cong \mathbb{C}[v, w, w^{-1}]/(w^3 - v^2, 1 - v w) \\ 
    \cong \mathbb{C}[v, w, w^{-1}]/(w^5 - 1, w^{-1} - v) \cong \mathbb{C}[w, w^{-1}]/(w^5 - 1)\text{.}
  \end{multline*}
  We get a reduced scheme of length five, so the rank is at most five.

  Now we show that $\brr(F) = 2$. Consider the equations given by rank one reflexive sheaves $\mathcal{O}(3,0)$
  and $\mathcal{O}(3,1)$ (given by minors of matrices as in the proof of Proposition \ref{proposition:closed_set}). In order to find these equations,
  we give coordinates to every point $p \in H^0(X_\Sigma, \mathcal{O}(6,0))^*$:
  \begin{align*}
    p = t_{6,0,0}x_0^6 + t_{0,6,0}x_1^6 + t_{0,0,6}x_2^6 + t_{4,1,1} x_0^4 x_1 x_2 
    + t_{1,4,1}x_0 x_1^4 x_2 + t_{1,1,4}x_0 x_1 x_2^4 \\
    + t_{3,3,0}x_0^3x_1^3 + t_{0,3,3}  x_1^3 x_2^3 + t_{3,0,3}x_0^3 x_2^3 + t_{2,2,2} x_0^2 x_1^2 x_2^2 \text{.}
  \end{align*}
  Now we write down the matrix of the map $(\cdot \apolar p) : S_{(3,0)} \to T_{(3,0)}$ in the standard monomial bases $\alpha_0^3, \alpha_1^3,
  \alpha_2^3, \alpha_0\alpha_1\alpha_2$ and $x_0^3, x_1^3, x_2^3, x_0 x_1 x_2$:
  \begin{equation*}
    M = \left(
  \begin{matrix}
    t_{6,0,0} & t_{3,3,0} & t_{3,0,3} & t_{4,1,1} \\
    t_{3,3,0} & t_{0,6,0} & t_{0,3,3} & t_{1,4,1} \\
    t_{3,0,3} & t_{0,3,3} & t_{0,0,6} & t_{1,1,4} \\
    t_{4,1,1} & t_{1,4,1} & t_{1,1,4} & t_{2,2,2}
  \end{matrix}
  \right)
  \end{equation*}
  We also write down the matrix of the map $(\cdot \apolar p) : S_{(3,1)}\to T_{(3,-1)}$ in the bases $\alpha_0^2 \alpha_1, \alpha_1^2 \alpha_2,
  \alpha_2^2 \alpha_0$ and $x_0^2 x_2, x_1^2 x_0, x_2^2 x_1$:
  \begin{equation*}
    N = \left(
  \begin{matrix}
    t_{4,1,1} & t_{2,2,2}& t_{3,0,3} \\
    t_{3,3,0} & t_{1,4,1} & t_{2,2,2}\\
    t_{2,2,2} & t_{0,3,3} & t_{1,1,4}
  \end{matrix}
  \right)
  \end{equation*}
  We compute that the $3$ by $3$ minors of $M$ and $N$ define an irreducible variety of dimension $5$ over $\mathbb{Q}$. But it can be found by the
  same method as in Subsection \ref{subsection:hirzebruch_surface} that the dimension of the second secant variety of the embedding $X_\Sigma
  \hookrightarrow \mathbb{P}(H^0(X_\Sigma, \mathcal{O}(6,0))^*)$ is $5$. Hence, the $\sigma_2(X_\Sigma)$ is given set-theoretically by the $3$ by $3$
  minors of $M$ and $N$ over $\mathbb{Q}$. But this means that it is also defined by these equations over $\mathbb{C}$. Finally, since $F$ satisfies
  these equations, the claim follows.
\end{example}

\begin{example}
  Now take $F = x_0^2 x_1^2 x_2^2 \in H^0(X_\Sigma, \mathcal{O}(6,0))^*$. Here the apolar ideal is $F^\perp = (\alpha_0^3, \alpha_1^3, \alpha_2^3)$.
  We calculate the following
    \begin{center}
      \begin{tabular}{|c|c|c|}
	\hline
	 $\rr(F)$ & $\crr(F)$ & $\brr(F)$ \\
	\hline
	$3$ & $ 3 $ & $3$  \\
	\hline
      \end{tabular}
    \end{center}
  Let $I = (\alpha_0^3 - \alpha_1^3, \alpha_1^3 - \alpha_2^3)$. In this case also no coordinate can be zero, so we may calculate the length on
  $U_\sigma$ (where $\sigma$ is as before). We get the ring as in Equation \ref{equation:length_calculation} and the two generators become $1 - u$ and
  $u - v$. So here the quotient ring is
  \begin{equation*}
    \mathbb{C}[w]/(w^3 - 1)\text{.}
  \end{equation*}
  This means that the rank is at most three (notice that we get a reduced scheme).
    We can calculate the Hilbert function of $A_F = S/F^\perp$ (where $F = x_0^2 x_1^2 x_2^2$). We have $\dim_{\mathbb{C}} (A_F)_{(3,1)} = 3$, so from
    Corollary \ref{corollary:border_catalecticant_bound} we get that $\brr(F) \geq 3$. 

    Now we show that $\crr(F) = 3$. We look at the polytope $P$ of the embedding by $\mathcal{O}(6,0)$.
    \[\begin{tikzpicture}
      \foreach \x in {-2,...,2}
      {
        \foreach \y in {-2,...,2} 
        {
          \draw[fill] (\x,\y)circle [radius=0.025];
        }
      }
      \fill[opacity=0.5, color=gray] (0,-2)--(2,2)--(-2,0)--cycle;
      \draw[-latex,thick] (0,-2) -- (1,0);
      \draw[-latex,thick] (0,-2) -- (0,-1);
      \draw[-latex,thick] (0,-2) -- (-1,-1);
      \draw[-latex,thick] (-2,0) -- (0,1);
      \draw[-latex,thick] (-2,0) -- (-1,-1);
      \draw[-latex,thick] (-2,0) -- (-1,0);
      \draw[-latex,thick] (2,2) -- (1,1);
      \draw[-latex,thick] (2,2) -- (1,0);
      \draw[-latex,thick] (2,2) -- (0,1);
      \node[right] at (2,2) {$(6,0,0)$};
      \node[below] at (0,-2) {$(0,0,6)$};
      \node[left] at (-2,0) {$(0,6,0)$};
      \node[below] at (0,0) {$(2,2,2)$};
    \end{tikzpicture}\]
    The projective tangent space at the vertex $v$ is given by the Hilbert basis of the semigroup $\mathbb{N}(P\cap M - v)$ (see Proposition
    \ref{proposition:embedded_tangent_space}). The vector $(2,2,2)$ is in none of the three Hilbert bases, which means that $x_0^2 x_1^2 x_2^2$ is in
    none the of three tangent spaces at the singular points. But the fact that $\brr(F) \geq 3$ means that $F$ is neither in any projective tangent
    space at a smooth point nor at any secant line passing through two points. It follows that $\crr(F) > 2$.
\end{example}

\bibliography{../../references}
\bibliographystyle{plain}

\end{document}